\documentclass[10pt]{article}
\usepackage[english]{babel}
\usepackage{amsmath,amsfonts,amssymb,amstext,amsmath,amsthm,amscd}
\usepackage{mathtools}
\usepackage[a4paper,top=2.5cm,bottom=3cm,left=3cm,right=3cm]{geometry}
\usepackage{breakcites}

\usepackage{booktabs}
\usepackage{mathtools}
\usepackage{mathrsfs}
\usepackage{dsfont}

\usepackage{graphicx}
\usepackage{subfigure}
\usepackage{wrapfig}
\usepackage{indentfirst}
\usepackage{latexsym}
\usepackage{sidecap}
\usepackage{booktabs}
\usepackage{hyperref}
\usepackage{lipsum}
\usepackage{cleveref}
\usepackage[dvipsnames]{xcolor}

\usepackage{tikz,tikz-3dplot}
\usepackage{ifthen}
\usetikzlibrary{patterns}% 

\newcommand\myshade{85}
\colorlet{mylinkcolor}{violet}
\colorlet{mycitecolor}{YellowOrange}
\colorlet{myurlcolor}{RoyalBlue}

\hypersetup{
    colorlinks=true,
    linkcolor=myurlcolor,
    citecolor=mycitecolor!\myshade!black,
    filecolor=magenta,      
    urlcolor=magenta,
}
\usepackage{setspace}

\usepackage{mathrsfs}
\usepackage{textcomp}
\usepackage{braket}
\usepackage{colortbl}
\usepackage{bbm}
\usepackage{comment}
\usepackage[colorinlistoftodos]{todonotes}\usepackage{mathrsfs}
\usepackage{textcomp}
\usepackage{braket}
\usepackage{colortbl}
\usepackage{bbm}
\usepackage{comment}
\usepackage[colorinlistoftodos]{todonotes}
\newcommand\xqed[1]{%
  \leavevmode\unskip\penalty9999 \hbox{}\nobreak\hfill
  \quad\hbox{#1}}
\newcommand\demo{\xqed{$\triangle$}}
\usepackage{enumitem}
\usepackage{musicography}

\usepackage{framed}

\theoremstyle{plain}
\newtheorem{theorem}{Theorem}[section]
\theoremstyle{definition}
\newtheorem{definition}{Definition}[section]
\theoremstyle{remark}
\newtheorem{remark}[theorem]{Remark}
\theoremstyle{remark}

\theoremstyle{plain}

\theoremstyle{plain}

\theoremstyle{plain}
\newtheorem{proposition}[theorem]{Proposition}
\theoremstyle{remark}
\newtheorem{example}{Example}[section]
\theoremstyle{remark}

\theoremstyle{remark}

\theoremstyle{remark}

\newcommand\RR{\mathbb{R}}

\newcommand\LL{\mathcal{L}}

\usepackage{comment}

\title{Weak energy shaping for stochastic controlled port--Hamiltonian systems}
\author{F. Cordoni, L. Di Persio and R. Muradore}
\date{ }      

\begin{document}

\maketitle   

\begin{abstract}
The present work address the problem of energy shaping for stochastic port--Hamiltonian system. Energy shaping is a powerful technique that allows to systematically find feedback law to shape the Hamiltonian of a controlled system so that, under a general passivity condition, it converges or tracks a desired configuration. Energy shaping has been recently generalized to consider stochastic port--Hamiltonian system. Nonetheless the resulting theory presents several limitation in the application so that relevant examples, such as the additive noise case, are immediately ruled out from the possible application of energy shaping. The current paper continues the investigation of the properties of a weak notion of passivity for a stochastic system and a consequent weak notion of convergence for the shaped system considered recently by the authors. Such weak notion of passivity is strictly related to the existence and uniqueness of an invariant measure for the system so that the theory developed has a purely probabilistic flavour. We will show how all the relevant results of energy shaping can be recover under the weak setting developed. We will also show how the weak passivity setting considered draw an insightful connection between stochastic port--Hamiltonian systems and infinite--dimensional port--Hamiltonian system.
\end{abstract}
\tableofcontents

\section{Introduction}\label{SEC:Intro}

In the last decades port--Hamiltonian systems (PHS) have seen a constantly growing interest. The theory of PHS merges two different points of view: (i) the theory of the port--based modelling and bond graphs, \cite{Bre1,Bre2,Dui}, aiming at providing a unified framework for physical systems belonging to different domains and (ii) Hamiltonian and geometric mechanics, \cite{Dal,Dal2}. Recently PHSs have been extensively used to tackle optimal control theory, \cite{OrtS,VdSC,Ort1}.

The main object in PHS theory is the \textit{Dirac structure}, that is a geometric object that describes the geometry of the system. \textit{Dirac structures} have been introduced in \cite{Cou} as a general geometric tool to treat degenerate symplectic structures in a unified way. Such objects allow to study and characterize the geometry of a wide variety of physical systems, that encompass pre-symplectic manifolds, Poisson dynamics and constrained systems, \cite{Dal,VdSC2,VdS10,Dal,Dal2}. The \textit{Dirac structure} defines an \textit{implicit Hamiltonian system}, leading to a definition of a Hamiltonian systems in terms of a set of algebraic--differential equations. Such a general description of a physical system allows to a systematic investigation of the interconnection, \cite{Mas,Dal2}, integrability, \cite{Dal} and symmetries, \cite{Bla}, and also to study physical systems with nonholonomic constraints, \cite{GB}. Recently, PHSs have been extended to the stochastic case, \cite{CDPMSPHS,CDPM_Tank,CDPM_IFAC,CDPM_Dis,Sat1,Sat2,Sat3,Sat4}.

Among the most relevant application of PHS is the usage of the geometric properties of interconnected systems to design suitable controls to achieve a precise goal, typically with the aim of stabilizing the overall system at a desidered configuration, or to track a desired trajectory, \cite{Ort1,OrtS}. In fact, many physical systems rest at a configuration in which their total energy function assumes a minimum. In case dissipation is present, such configuration is asymptotically stable. Rarely the minimum of the potential energy coincides with the desired configuration, so that the idea is to implement proper control actions able to shape the system energy in order to force a minimum in correspondence with the desired configuration.
This control technique is called \textit{energy shaping}, \cite{Ort1,Fan}. The stabilization of the system follows from the \textit{passivity} property together with the La Salle's invariance principle.

The generalization of \textit{energy shaping} techniques to the stochastic case is non trivial. Stabilization of stochastic passive systems has been first studied in \cite{Flo1,Flo2,Flo3,Sat2}. Energy shaping for stochastic PHS (SPHS) has been addressed in \cite{Had}, where standard results from deterministic energy shaping have been adapted by considering stochastic PHS. In \cite{Sat1,Sat3} different notions of stochastic stabilization are considered to include a broad range of possible physical examples. A different and yet related approach to stochastic energy shaping via \textit{Casimir functionals}, that conserved quantities of the system, is studied in \cite{Arn}. The authors addressed energy shaping of a class of stochastic Hamiltonian systems via the associated \textit{infinitesimal generator}. As noted by the authors, their approach is only valid for short time, whereas to look at the long--time behaviour the \textit{invariant measure} of the stochastic system must be considered. In \cite{Fang,CDPM_Tank,CDPM_IFAC} a \textit{weak notion of stochastic stability} is considered, showing that such notion is in turn strictly related to the \textit{invariant measure} of the SPHS. Broadly speaking, the weak notion of passivity introduced in \cite{Fang,CDPM_Tank,CDPM_IFAC} is not defined on the whole state--space but only outside a ball centred at a specified state. This definition has several desirable implications regarding the limiting distribution of the system. It turns out that this weak notion of passivity is tailor-made to deal with stochastic equations with additive noise, allowing to extend previous results to consider also the case of SPHS's with non vanishing noise. 

The present work extensively and systematically studies the notion of weak stochastic passivity used in \cite{CDPM_Tank} and the related convergence, with particular attention to the energy shaping of SPHS. We will show that the proposed approach generalizes the results in \cite{Had}, including relevant physical systems that do not fall in their assumption. In particular, results proved in \cite{Had}, although being a very interesting first step in generalizing energy shaping to a stochastic scenario, have a few practical as well as theoretical limitations. On one side, in order to design the control, \textit{strong Casimir} are considered. By strong Casimir, following the notation of \cite{CDPMSPHS,Ort}, we mean $\mathbb{P}-$a.s. conserved quantities. On the other side, given the notion of convergence used in \cite{Had}, in order to stabilize the system the noise must vanishes at the desired configuration. The usage of \textit{strong Casimir}, implies that the control must share the same noise as the physical system to be controlled. This is needed since, in order to obtain a Casimir for the system, the noise of the control must compensate $\mathbb{P}-$a.s. the noise of the system. Such condition, as shown in \cite{CDPMSPHS}, is hardly satisfied in real applications since it implies that it is possible to separate at any time the state of the system from the noise. Also, it is worth stressing that strong Casimir functional are rare and difficult to obtain. 

The assumption on the vanishing noise has also other restrains. In fact, the idea of the energy shaping approach for PHS is that a controller can in principle stabilize the system around any configuration designing a suitable control law in feedback form so that the resulting controlled PHS is again a PHS with a new Hamiltonian function having a minimum in the desired configuration. This fact, together with the passivity property of the PHS, implies that the system stabilizes around the minimum of the Hamiltonian. Since the fact that the control does not affect the noise of a SPHS, the randomness of a SPHS cannot be changed by any law. Therefore, since the notion of stochastic stability used in \cite{Had} requires a vanishing noise, it turns out that a system can be stabilized only around configurations for which the noise vanishes. Such assumption strongly limits the possible configurations around which a SPHS can be stabilized. Even more relevant, above assumption immediately rule out additive noise which is the standard case when real sensors are taken into account.

The approach proposed in the current paper solves all of the above problems. Using a weak notion of passivity we are able to consider a weak notion of convergence, which is strictly related to the invariant measure of a SPHS. Such notion, without any requirement on the vanishing noise, allows to include additive noise as well as to stabilize the system around any configuration. In the case of a vanishing noise we recover stabilization as proved in \cite{Had}. Further, control design can be done as in the deterministic case where now the stochastic system oscillates around the desired configuration with a magnitude given by the noise that affects the system. Also, using the notion of weak Casimir we are able to include a wider class of possible Casimir. In order to be as general as possible, we will prove the main results also for the	 relevant class of stochastic systems with degenerate noise.

We will also show how, as typical in stochastic analysis, the problem of finding an invariant measure for a SPHS can be solve looking at stationary solutions for a deterministic PDE, called the \textit{Fokker--Planck equation}, \cite{Lor,Bor_book}. We will show that such deterministic PDE can be proved to be an infinite dimensional PHS in a Lebesgue space weighted by the invariant measure. Such result has a major consequence. It provides a deep and interesting connection between weak energy shaping of SPHS and energy shaping for infinite--dimensional PHS. This allows to tackle the problem of energy shaping either from a deterministic or a stochastic point of view. Such connection is only introduced in the current work and it will be further studied in the future.
% object of future deep investigation, mostly in the development of the so--called H--theorem, \cite{Bar,FraHTH,Jau,Chav}, in the context of PHS, which to the best of our knowledge has never been investigated.

The main contributions of the present paper are:
\begin{description}
\item[(i)] to study a weak notion of stochastic passivity and stability for a wide class of SPHS;
\item[(ii)] to investigate the problem of energy shaping via the new proposed notion of stochastic passivity;
\item[(iii)] to generalize the energy shaping approaches available in the literature;
\item[(iv)] to show that the Fokker--Planck equation associated to a SPHS can be seen as an infinite--dimensional deterministic PHS. 
\end{description}

The structure of the paper is as follow: in Section \ref{SEC:SPSDE} we introduce the main notions of stochastic passivity and stability. In Section \ref{SEC:WeakPass} we introduced rigorously the weak notion of stochastic passivity, while in Section \ref{SEC:ConnInf} we prove the connection to infinite--dimensional PHS. Section \ref{SEC:EnSh} is devoted to studying energy shaping under the weak notion of stochastic passivity introduced; Section \ref{SEC:Example} shows two examples where explicit invariant measures for a SPHSs are calculated. 

\section{Stability and passivity for stochastic differential equations}\label{SEC:SPSDE}

Throughout the work we will consider a complete filtered probability space $\left (\Omega,\left (\mathcal{F}_t\right )_{t\geq 0},\mathbb{P}\right )$ satisfying usual assumptions. 

Before entering into details on energy shaping for stochastic port--Hamiltonian systems (SPHS), in order to make the paper as much self--contained as possible, some key results regarding the stability of a general \textit{stochastic differential equation} (SDE) are briefly recalled. Consider a stochastic process $\left (X(t)\right )_{t \geq 0} \in \RR^n$ satisfying the following SDE
\begin{equation}\label{EQN:SDE}
\begin{cases}
dX(t)=\mu(X(t))dt+\sigma(X(t))dW(t)\,,\\
X(s)=x\,,
\end{cases}
\end{equation}
where $\mu:\RR^n\to\RR^n$ and $\sigma:\RR^n \to \RR^{n \times d}$ are suitable regular enough coefficients, $W(t)$ is a $d-$dimensional standard Brownian motion and $dW(t)$ is the integration in It\^{o} sense. We will use the convention $X^{s,x}(t)$ to denote the solution of equation \eqref{EQN:SDE} at time $t$ starting at time $s<t$ with initial value $x$. If no confusion is possible we will write for short $X(t) = X^{s,x}(t)$.

Next we recall different possible notions of convergence for a stochastic process $X$, \cite{Flo3,Kha}.

\begin{definition}\label{DEF:Stab}
The equilibrium solution $X(t) \equiv 0$ is said:
\begin{description}
\item[(i)] \textit{stable in probability} if for any $s \geq 0$ and $\epsilon>0$ it holds
\[
\lim_{x\to 0} \mathbb{P}\left (\sup_{s \leq t} \left |X^{s,x}(t) \right | > \epsilon \right )=0\,;
\]

\item[(ii)] \textit{locally asymptotically stable in probability} if it is \textit{stable in probability} and for any $s \geq 0$ it holds
\[
\lim_{x\to 0} \mathbb{P}\left (\lim_{t \to \infty} \left |X^{s,x}(t) \right | =0 \right )=1\,;
\]

\item[(iii)] \textit{asymptotically stable in probability} if it is \textit{stable in probability} and for any $s \geq 0$ and $x \in \RR^n$ it holds
\[
\mathbb{P}\left (\lim_{t \to \infty} \left |X^{s,x}(t) \right | =0 \right )=1\,.
\]
\end{description}
\end{definition}

We will denote by $\mathcal{L}$ the \textit{infinitesimal generator} of the process \eqref{EQN:SDE}. Recall that, \cite{Kar}, for $f:\RR^n\to \RR$ regular enough the \textit{infinitesimal generator} $\mathcal{L}$ of the process $X$ satisfying equation \eqref{EQN:SDE} is defined as 
\[
\mathcal{L}f(x):= \lim_{t \to 0} \frac{\mathbb{E}_{t,x}[f(X(t))] - f(x)}{t}\,,
\]
being $\mathbb{E}_{t,x}$ the conditional expectation w.r.t. $t$ and $x$. It can be shown, \cite{Kar}, that the \textit{infinitesimal generator} of $X$ satisfying equation \eqref{EQN:SDE} is explicitly given by
\begin{equation}\label{EQN:InfG}
\begin{split}
\mathcal{L}f(x) &= \sum_{i=1}^n \mu_i(x) \partial_{x_i}f(x) + \frac{1}{2}\sum_{i,j=1}^n \left (\sigma(x)\sigma^T(x)\right )_{ij} \partial^2_{x_i \, x_j}f(x) = \\
&= \mu(x) \cdot \partial_x f(x) + \frac{1}{2}Tr\left [\sigma(x)\sigma^T(x)\partial_x^2 f(x)\right ]\,,
\end{split}
\end{equation}
where $\partial_x$ and $\partial_x^2$ are the first and second derivative in $x$, respectively.

Stability of a SDE can be inferred assessing certain properties of the \textit{infinitesimal generator} $\mathcal{L}$. In particular, the following stochastic Lyapunov theorem holds, \cite{Flo3,Kha}.

\begin{theorem}
Assume there exists a Lyapunov function $V \in C^2(D;\RR)$ positive definite in a bounded open set $D$ of $\RR^n$ containing the origin. If, for any $x \in D \setminus\{0\}$,
\begin{equation}
\mathcal{L}V(x) \leq 0\,,\quad \mbox{resp.} \quad \mathcal{L}V(x) < 0\,,
\end{equation}
then the equilibrium solution $X(t) \equiv 0$ of the SDE \eqref{EQN:SDE} is \textit{stable in probability}, resp. \textit{locally asymptotically stable}, in probability.

If further $D = \RR^n$, the Lyapunov function $V$ is said to be proper and the stability to be global.
\end{theorem}

\subsection{On the passivity for controlled SDE}\label{SEC:Pass}

The present section is devoted to the introduction of the concept of passivity for controlled SDE and on its relation to stochastic stability.

Consider an \textit{input--state--output} stochastic process $\left (X(t)\right )_{t \geq 0}\in \RR^n$ satisfying the SDE
\begin{equation}\label{EQN:SDEC}
\begin{cases}
dX(t)=\mu(X(t),u(t))dt+\sigma(X(t))dW(t)\,,\\
y(t)=h(X(t),u(t))\,,
\end{cases}
\end{equation}
for $u \in \mathcal{U}$, being the space of all $\left (\mathcal{F}_t\right )_{t\geq 0}-$adapted process $u:[0,T]\to U$ so that $u \in L^2([0,T])$ $\mathbb{P}-$a.s. is a $U-$valued progressively measurable process, where $U$ is a closed subset of $\RR^m$, representing the domain of the control process acting on the state process $X$. Also, $y \in \mathcal{Y}$, being the space of all $\left (\mathcal{F}_t\right )_{t\geq 0}-$adapted process $\mathbb{P}-$a.s. is a $Y-$valued progressively measurable process, being $Y$ a closed subset of $\RR^m$.

We will denote by $X^0(t)$ the solution to the autonomous system \eqref{EQN:SDEC}, that is the solution with constant null control $u \equiv 0$, while, $\mathcal{L}_0$ is the infinitesimal generator of the autonomous system \eqref{EQN:SDEC}. Further, for a suitable regular enough function $f$, $\partial_x f(x,u)$ denoted the partial derivative w.r.t. the first argument whereas $\partial_u f(x,u)$ denotes the partial derivative w.r.t. the second argument.

Next is the definition of stochastic passivity, \cite[Definition 3.1]{Flo3}.

\begin{definition}\label{DEF:Pass}
The input--state--output system \eqref{EQN:SDEC} is said to be \textit{passive} if there exists a Lyapunov function $V$ on $\RR^n$, called \textit{storage function}, such that
\begin{equation}
\mathcal{L}V(x) \leq h^T(x,u) u\,,
\end{equation}
for every $(x,u) \in \RR^n \times U$.
\end{definition}

A straightforward application of It\^{o} formula and \textit{Dynkin lemma}, \cite{Kar}, shows that \textit{stochastic passivity}, as defined in Definition \ref{DEF:Pass}, implies that
\[
\mathbb{E}V(X(t))\leq V(x) + \mathbb{E} \int_0^t h^T(X(s),u(s)) u(s)ds \,.
\]

Strictly related to stability of a SDE, the following necessary conditions for the input--state--output system \eqref{EQN:SDEC} to be passive can be state, \cite{Flo3}. In the following, we will denote by $\mathcal{L}_0$ the \textit{infinitesimal generator} of the autonomous system \eqref{EQN:SDEC}, i.e. $u \equiv 0$.

\begin{theorem}
The following conditions are necessary for system \eqref{EQN:SDEC} to be passive:
\begin{description}
\item[(i)] $\mathcal{L}_0 V(x) \leq 0$, for every $x \in \RR^n$;

\item[(ii)] for every $x \in \mathcal{S} := \{ x \in \RR^N \, : \, \mathcal{L}_0 V(x) = 0 \}$ it holds
\[
\sum_{i=1}^n \partial_u \mu_i(x,0) \partial_{x_i} V(x) = h^T(x,0)\,;
\]

\item[(iii)] for every $x \in \mathcal{S} := \{ x \in \RR^N \, : \, \mathcal{L}_0 V(x) = 0 \}$ it holds
\[
\sum_{i=1}^n \partial_{u \,u}^2 \mu_i(x,0) \partial_{x_i} V(x) \leq \partial_u h^T(x,0) + \partial_u h(x,0)\,.
\]
\end{description}
\end{theorem}

In the particular case that the stochastic system \eqref{EQN:SDEC} is affine in the control, that is
\begin{equation}\label{EQN:SDECAff}
\begin{cases}
dX(t)=\left (\mu(X(t)) + \bar{\mu}(X(t))u(t) \right )dt+\sigma(X(t))dW(t)\,,\\
y(t)=h(X(t))
\end{cases}
\end{equation}
the following stochastic version of the \textit{Kalman--Yakubovich--Popov} (KYP) property can be proven, \cite{Flo1}.

\begin{definition}\label{DEF:KYP}
The stochastic system in affine form \eqref{EQN:SDECAff} satisfies the KYP property if there exists a proper Lyapunov function $V$ such that for every $x \in \RR^n$, it holds
\[
\begin{cases}
\LL_0 V(x) \leq 0\,,\\
\sum_{i=1}^n \partial_{x_i}V(x) \bar{\mu}(x) = h^T(x)\,.
\end{cases}
\]
\end{definition}

We thus have the following.

\begin{theorem}
The stochastic system \eqref{EQN:SDECAff} is passive if and only if it satisfies the KYP property.
\end{theorem}

\subsection{On the passivity for controlled stochastic port--Hamiltonian system}\label{SEC:PassSPHS}

Having introduced the main notations and results regarding stability and passivity of general $\RR^n-$valued SDE, we can thus introduce the notion of stochastic port--Hamiltonian system.

A stochastic PHS $\left (X(t)\right )_{t \geq 0}$ is the solution to the SDE
\begin{equation}\label{EQN:SPHS}
\begin{cases}
dX(t) = \left [(J-R)\partial_x H(X(t)) + g(X(t)) u(t)\right ]dt + \sigma(X(t)) dW(t)\,,\\
y(t)=g^T(X(t)) \partial H(X(t))\,,
\end{cases}
\end{equation}
being $J=J^T$ a given $n\times n$ matrix, $R \succeq 0$ a $n\times n$ positive semi--definite matrix representing dissipation, $H$ the \textit{Hamiltonian} of the system, $u \in \mathcal{U}$ the control and $y \in \mathcal{Y}$ the output of the system. As above, $W$ is a standard Brownian motion and $ dW(t)$ denoted the integration in the sense of It\^{o}. 

\begin{remark}
It is worth remarking that in \cite{CDPMSPHS} a stochastic port--Hamiltonian system has been defined in terms of the Stratonovich integral; the present work adopt the stochastic integration on the It\^{o} sense as passivity is typically addressed considering It\^{o} notion of integration. In \cite{CDPM_Tank}, an energy tank approach for a teleoperated system modelled as SPHS has been investigated. Similarly to the present paper, in \cite{CDPM_Tank}, the It\^{o} point of view has been considered, since the main object of investigation was the passivity of the system. Nonetheless, in that paper, it has been shown how the It\^{o} SPHS can be converted into a corresponding Stratonovich SPHS, showing further how passivity is affected.

\demo
\end{remark}

Using \textit{Dynkin lemma} and It\^{o} formula, it can be seen that the SPHS \eqref{EQN:SPHS} satisfies the energy preserving property
\begin{equation}\label{EQN:EnPres}
H(X(t_2))-H(X(t_1))= \int_{t_1}^{t_2} \mathcal{L}H(X(s)) ds + \int_{t_1}^{t_2} \partial^T H(X(s)) \sigma(X(s))  dW(s)\,,
\end{equation}
being $\mathcal{L}$ the infinitesimal generator \eqref{EQN:InfG} for the It\^{o} SPHS \eqref{EQN:SPHS} defined as
\begin{equation}\label{EQN:IngGSPHS}
\mathcal{L}H(x) = \left [\left (J-R\right )\partial_x H(x) + g(x)u\right ] \cdot \partial_x H(x) + \frac{1}{2}Tr\left [\sigma(x)\sigma^T(x)\partial_x^2 H(x)\right ]\,.
\end{equation}

We thus have the following result concerning the passivity and convergence of a SPHS, \cite[Theorem 4]{Had} or also \cite{Flo3,Sat2}.

\begin{proposition}\label{PRO:PassSPHS}
Consider the stochastic PHS \eqref{EQN:SPHS}, if
\begin{equation}\label{EQN:PassCond1}
2\partial_x^T H(x) R(x) \partial_x H(x) \geq Tr\left [\partial_x^2H(x) \sigma(x)\sigma^T(x)\right ]\,,
\end{equation}
then, the SPHS \eqref{EQN:SPHS} is passive. 

Further if:
\begin{itemize}
\item[(i)] the SPHS \eqref{EQN:SPHS} is passive;
\item[(ii)] the noise vanishes at an equilibrium configuration, that is $\sigma(x_e)=0$;
\end{itemize} 
then the equilibrium solution $X(t) \equiv x_e$ is Lyapunov stable in probability. If, in addition
\[
\left \{ x \in \RR^n \,:\, \mathcal{L}_0 H(x) = 0\right \} = \{x_e\}\,,
\]
then the equilibrium solution $X(t) \equiv x_e$ is locally asymptotically stable in probability.
\end{proposition}
\begin{proof}
Using the infinitesimal generator \eqref{EQN:InfG}, it immediately follows that condition \eqref{EQN:PassCond1} yields
\[
\mathcal{L} H(x) \leq y^T(t)u(t)\,,
\]
which is the Definition \ref{DEF:Pass} of passivity for the SPHS \eqref{EQN:SPHS}.

The convergence thus follows using \cite[Thm. 5.3, Cor. 5.1, Thm. 5.11]{Kha}.
\end{proof}

\begin{remark}
Above results highlight how classical notions of stochastic passivity and stability have clear weaknesses. In particular, a vanishing noise need to be required, that is $\sigma(x_e)=0$ at the desired equilibrium state $x_e$. Such assumption is in general strong, but in the context of energy shaping for SPHS this assumption can have even stronger implications. In particular, it is a system property and cannot be modified in any way suitably shaping the energy of the system. The main idea of energy shaping is to derive a feedback control law $u = \phi(X)$ so that the dynamics of the stochastic equation under the feedback law preserves the port--Hamiltonian structure. The final goal is to shape the energy of the system so that it can be stabilized at a certain state $x_e$, which was not the minimum of the original Hamiltonian function. Nonetheless, since only the Hamiltonian function can be suitably shaped and in particular the noise cannot be modified in any way, this new equilibrium point should already be a point for which the noise vanishes from the beginning.

\demo
\end{remark}

\section{Ultimately stochastic passivity and stability for controlled stochastic port--Hamiltonian system}\label{SEC:WeakPass}

In the introduction the main limitations behind the classical notion of passivity have been briefly mentioned, explaining how we intent to weaken certain approaches to suitably extend the idea of energy shaping to a broader stochastic context. The present section is devoted to introducing a weak notion of stochastic passivity and a consequently related notion of stability, that appears to be tailor--made for tackling the problem of energy shaping for SPHS. Such a notion has been first introduced in \cite{Fang} and already used in the SPHS in \cite{CDPM_Tank,CDPM_IFAC}.

The definition of \textit{ultimately stochastic passivity} for a stochastic port--Hamiltonian system, \cite{CDPM_IFAC} is recalled.
\begin{definition}\label{DEF:USP}[Ultimately stochastic passivity]
The stochastic PHS \eqref{EQN:SPHS} is said to be \textit{ultimately stochastic passive} if for $x_e \in \RR^n$ and for any $x$ such that $\|x-x_e \| \geq C$, for a given constant $C>0$ called passivity radius, it holds
\[
\mathcal{L}H(x) \leq y^T u\,.
\]

If further there exists $\delta_C > 0$ such that for $\| x-x_e \| \geq C$ it holds
\[
\mathcal{L}H(x) \leq y^T u - \delta_C \| x-x_e \|^2 \,,
\]
then system \eqref{EQN:SPHS} is said to be \textit{strictly ultimately stochastic passive}.
\end{definition}

As mentioned in \cite{CDPM_Tank,CDPM_IFAC}, to highlight the connection of weak passivity with the limiting invariant measure and with the notion of deterministic ultimately bounded process, we will use the name \textit{ultimately stochastic passivity} instead of \textit{weak passivity}. In particular, our choice is motivated by the fact that the concept of \textit{weak passivity} is closely related to a more general notion of convergence in the deterministic setting and called \textit{ultimately bounded}. In fact, in the presence of a non-vanishing term, as in the present context with an additive noise, the process does not converge to an equilibrium but instead it can be proven to be bounded in a suitable domain. In the deterministic setting, the bounded stability can be proven to hold if there exists a Lyapunov function $V$, so that $\dot{V}(x)<0$, $\forall$ $x$ such that $\|x-x_e\|>C$, with a given $x_e$. In the stochastic setting we could retrieve a similar result choosing as candidate Lyapunov function the Hamiltonian of the system. If the process $X$ is \textit{ultimately stochastic passive}, or equivalently in the terminology of \cite{Fang} \textit{weakly stochastic passive}, then for the autonomous process with null control, i.e., $u \equiv 0$, it holds $\mathcal{L}V(x)<0$, $\forall$ $x$ such that $\|x-x_e\|>C$.

The notion of ultimately stochastic passivity can be thought as follows: for $X$ converging to $x_e=0$ we have that the system is not passive, as the noise keeps injecting energy into the system preventing the system from asymptotically stabilizing at $x_e=0$. Nonetheless, the system cannot exhibits non stationary behaviours since, as soon as the process exits a suitable ball of radius $C$, the system becomes passive and the energy injected by the noise into the system is strictly less then the one dissipated, so that the system recovers stability. It follows that, at large time, the system will be forced to stay in a fixed domain and keeps oscillating around the stationary point $x_e=0$ according to a suitable invariant law.

\begin{remark}
As it will be clear in a while, the notion of \textit{strictly ultimately stochastic passivity} is fundamental in proving the existence and uniqueness of an invariant measure for a stochastic system. In general to ensure stability of a stochastic system it is not enough to require only \textit{ultimately stochastic passivity}, namely that there exists $C > 0$, so that for all $\|x\| \geq C$, it holds
\[
\mathcal{L} H(x) < 0\,.
\]

As a counterexample consider the system
\begin{equation}\label{EQN:ExampleUnst}
dX(t) = \frac{X(t)}{X^2(t)+1}dt + dW(t)\,,
\end{equation}
and as Lyapunov candidate the function
\[
H(x) = \log^2 (1+|x|)\,.
\]

The explicit computation shows that there exists $C>0$ so that, for $\|x\| \geq C$ it holds
\begin{equation}\label{EQN:CondInstab}
\mathcal{L} H(x) \leq 0\,.
\end{equation}

Nonetheless, system \eqref{EQN:ExampleUnst} diverges, hence becoming unstable. The key point is that, even if condition \eqref{EQN:CondInstab} holds true, it can be seen that
\[
\lim_{x \to \pm \infty} \mathcal{L} H(x) = 0\,.
\]
To avoid such a phenomenon, thus ensuring system stability, the correct requirement is that there exist $C>0$ and $\epsilon >0$, such that
\[
\mathcal{L} H(x) < - \epsilon\,.
\]
for $\|x\| \geq C$.

\demo
\end{remark}

Although Definition \ref{DEF:USP} might seem similar to the standard definition, it has some key aspects that makes it more suitable to be adapted to study SPHS and in general to address the problem of energy shaping. Some of these key features will be showed in the remaining of the current section, other will clearly emerges in subsequent sections.

A remarkable aspects of \textit{ultimately stochastic passivity} of SPHS, is that under general and relevant setting, it can be shown that a SPHS cannot be \textit{passive} but on the contrary it is always \textit{ultimately stochastic passive}. 
%Consider in fact $H$ to be a quadratic function of the state, i.e. $H(X(t)) = \frac{1}{2} X^T(t) \Lambda X(t)$, with $\Lambda$ a positive definite symmetric constant matrix of suitable dimensions and consider a constant additive noise, that is $\sigma(X(t)) \equiv \sigma$, being $\Sigma:= \sigma^T \sigma \succ 0$ a $n \times n$ positive definite matrix. Then, such simple and yet extremely relevant SPHS is never passive but it is always \textit{ultimately stochastic passive}.

\begin{proposition}\label{PRO:SPass1Add}
Consider the stochastic PHS \eqref{EQN:SPHS} with additive noise, i.e. $\sigma(X(t)) \equiv \sigma$, being $\Sigma:= \sigma^T \sigma \succ 0$ a $n \times n$ positive definite matrix, and $H_a$ quadratic Hamiltonian of the form $H(X(t)) = \frac{1}{2} X^T(t) \Lambda X(t)$, with $\Lambda$ a positive definite symmetric $n \times n$ matrix.

Then the SPHS is never \textit{passive} but it is always \textit{ultimately stochastic passive}.
\end{proposition}
\begin{proof}
Using the skew--symmetric property of the matrix $J$, the infinitesimal generator \eqref{EQN:InfG} of the It\^{o} process SPHS \eqref{EQN:SPHS} $\mathcal{L}$ is given by
\begin{align}
\mathcal{L}H(x) &= \partial_x^T H(x) \left [\left (J(x) - R(x) \right )\partial_x H(x) + g(x)u \right ] + \frac{1}{2}Tr\left [\partial_x^2 H(x)\Sigma \right ] = \\
&=-\partial_x^T H(x) R(x) \partial_x H(x) + \partial_x^T H(x) g(x)u + \frac{1}{2}Tr\left [\partial_x^2 H(x)\Sigma\right ]\,.
\end{align}

Using the quadratic form of the Hamiltonian function it follows that
\begin{align}
\partial_x^T H(x) R(x) \partial_x H(x) &= x^T \Lambda R(x) \Lambda x\,,\\
Tr\left [\partial_x^2 H(x)\Sigma \right ] &= Tr\left [\Lambda \Sigma \right ]\,.
\end{align}

Since $Tr\left [\Lambda \Sigma \right ]>0$ is constant and strictly positive, it is immediate to see that for $x$ sufficiently small, that is it exists $\epsilon >0$ such that for $\|x\|<\epsilon$, the \textit{stochastic passivity} is violated as
\begin{equation}\label{EQN:StrongPBound}
x^T \Lambda R \Lambda x < Tr\left [\Lambda \Sigma \right ]\,.
\end{equation}

On the contrary, there exists a positive constant $C>0$ such that, for $\|x\| \geq C$ it holds
\begin{equation}\label{EQN:WeakPBound}
x^T \Lambda R(x) \Lambda x \geq Tr\left [\Lambda \Sigma \right ]>0\,,
\end{equation}
implying that, for $\|x\| \geq C$, there exists $\delta_C > 0$ such that
\[
\mathcal{L}H(x) \leq \partial_x^T H(x) g(x)u = y^T(t) u(t) - \delta_C \|x\|^2\,,
\]
which is the Definition \ref{DEF:USP} of \textit{ultimately stochastic passivity}.
\end{proof}

A further immediate and relevant consequence of Proposition \ref{PRO:SPass1Add} is that, for the class of SPHS considered above, that is SPHS with additive noise and quadratic Hamiltonian, no additional conditions have to be assumed to guarantee the \textit{ultimately stochastic passivity} of the stochastic system. This is in contrast to the classical notion of stochastic passivity, where an additional condition compared to the standard deterministic setting must be imposed. In fact, if \textit{ultimately stochastic passivity} is considered, the SPHS considered in Proposition \ref{PRO:SPass1Add} is passive under the typically condition $R \succ 0$.

The notion of \textit{ultimately stochastic passivity} is strictly related to a convergence in a suitable weak sense of the SPHS. In particular, we will show that if the SPHS is \textit{strictly ultimately stochastic passive}, then it converges toward the unique invariant measure of the system. 

In order to prove the convergence of the SPHS toward an invariant measure, the Feller property and the transition Markov semigroup for the SPHS have to be studied. We will first recall basic definitions and results about ergodicity and Markov property for stochastic differential equations; we refer the reader to \cite{Dap,Bor_book,Kha} for further details.

In the following we will assume without loss of generality that the SPHS \eqref{EQN:SPHS} is equipped with the initial condition $X(s)=x$; we will denote for short by $X(t) \equiv X^{s,x}(t)$ the solution of the SPHS \eqref{EQN:SPHS} at time $t$ with initial time $s<t$ and initial state $x \in \RR^n$. We will say that the SPHS is a Markov process on $\RR^n$, if
\[
\mathbb{P}\left (\left .X(t) \in B\right | \mathcal{F}_s\right ) = \mathbb{P}\left (\left .X(t) \in B \right |X(s)\right )\,,\quad \mathbb{P}-\mbox{a.s.}\,,
\]
for all $t \geq s$ and Borel set $B \in \mathcal{B}(\RR^n)$. In the following we will consider time homogeneous Markov process so that the considered process is invariant up to a time rescaling. This means that it is equivalent to consider as initial time $s=0$; for this reason in the following we will omit explicitly the dependence upon the initial time $s$. For any \textit{Markov process} we can introduce the notion of \textit{Markov transition function} $p(t,x,B)$, namely
\[
\mathbb{P}(X^x(t) \in B) =: p(t,x,B)\,,\quad B \in \mathcal{B}(\RR^n)\,.
\]

We will say that the transition semigroup $p$ is called \textit{Feller semigroup} if the \textit{Markov semigroup} $P_t$ defined as
\[
P_t f(x) := \mathbb{E} f(X^x(t))\,,
\] 
is bounded and continuous for any $f \in C_b(\RR^n)$, being $C_b(\RR^n)$ the space of bounded and continuous function on $\RR^n$. If $P_t f(x)$ is continuous and bounded for any $t > 0$ and for any $f \in C_b(\RR^n)$, then it is called \textit{strongly Feller semigroup}. If, for $t >0$, all \textit{Markov semigroup} $P_t$ are equivalent, then $P_t$ is called $t-$regular.

The \textit{Markov semigroup} and the \textit{Markov transition function} are connected by the following
\begin{equation}\label{EQN:MarkovSp}
P_t f(x) = \int_{\RR^n} f(y)p(t,x,dy)\,,
\end{equation}
that can be also expressed as
\[
P_t \mathbbm{1}_{B}(x) = p(t,x,B)\,, \quad B \in \mathcal{B}(\RR^n))\,.
\]

Under certain regularity condition, \cite{Lor}, the function $v(t,x)$ defined through \textit{Markov semigroup} in equation \eqref{EQN:MarkovSp} as
\[
v(t,x) := P_t f(x)\,,
\]
is the solution of the Cauchy problem
\begin{equation}\label{EQN:Cauchy}
\begin{cases}
\partial_t v(t,x) - \mathcal{L}v(t,x) = 0\,,\quad (t,x) \in \RR_+ \times \RR^n\,,\\
v(0,x) = f(x)\,,
\end{cases}
\end{equation}
where $\mathcal{L}$ is the infinitesimal generator of the Markov process with \textit{transition function} $p(t,x,B)$.

For the definitions of \textit{strongly Feller Markov semigroup} and \textit{regular Markov semigroup}, that will be used later, we refer to the literature \cite{Dap}.

We can give the following definition of \textit{invariant measure} for a SPHS, \cite[Definition 1.5.14]{Bor_book}.

\begin{definition}\label{DEF:Inv}
Consider the SPHS \eqref{EQN:SPHS}, a measure $\rho$ is said to be an \textit{invariant measure} for the SPHS \eqref{EQN:SPHS} if it holds
\begin{equation}\label{EQN:Inv}
\int_{\RR^n} p(t,x,B)\rho(dx) = \rho(B)\,.
\end{equation}
\end{definition}

The Definition \ref{DEF:Inv} can be equivalently written in terms of the Markov transition semigroup $P_t$ as
\begin{equation}\label{EQN:InvMarkovP}
\int_{\RR^n} P_t f(x)\rho(dx) = \int_{\RR^n} f(x) \rho(dx)\,,\quad f \in B_b(\RR^n)\,,
\end{equation}
being $B_b(\RR^n)$ the set of Borel and bounded function over $\RR^n$.

If the process $X$ is a Feller process, then it can be shown that, \cite[Lemma 2.6.14]{Bor_book}, a measure $\rho$ is invariant according to the Definition \ref{DEF:Inv} if and only if it is \textit{infinitesimally invariant}, that is,
\[
\int_{\RR^n} \mathcal{L}f(x)\rho(dx) = 0\,,
\]
for any $f$ in the domain of the \textit{infinitesimal generator} $\mathcal{L}$ of $X$.

Further, if the transition semigroup $P_t$ admits an invariant measure $\rho$, then the semigroup can be extended to a semigroup of bounded operators in the space $L^p(\RR^n,\rho)$, \cite[Corollary 8.1.7, Proposition 8.1.8]{Lor}. Such characterization allows to investigate the long-time behaviour of the semigroup; in particular it holds that
\[
\lim_{t \to \infty} \|P_t f - \bar{f}\|_p = 0\,,\quad \bar{f}:= \int_{\RR^n} f \rho(dx)\,,
\]
which corresponds to the long-time behaviour of equality \eqref{EQN:Cauchy}.

For what concern existence of an invariant measure, exploiting the \textit{ultimately stochastic passivity} property of the system, we will show that there exists a unique invariant measure. In particular, the existence follows from the next result, \cite[Prop. 3.1]{Las}, that we report in order to make the treatment as much self-contained as possible. 

\begin{proposition}[Proposition 3.1 \cite{Las}]\label{PRO:Las}
Let $\left (X,\| \cdot\|_X\right )$ be a complete separable metric space and let $\left (P_t\right )_{t \geq 0}$ be the semigroup of Markov operators corresponding to a Markov process which satisfies the Feller property. Assume that there exist a compact set $B$ and a point $x \in X$ such that
\[
\lim \sup_{T \to \infty} \left (\frac{1}{T}\int_0^T P_t \mathbbm{1}_{B}(x)dt \right ) > 0\,.
\]
Then the Markov process has a stationary distribution.
\end{proposition}

We thus have the following notion of convergence for a SPHS.

\begin{definition}\label{DEF:WConv}
Consider the SPHS \eqref{EQN:SPHS} and assume that it admits an invariant measure $\rho$. If for any Borel set $B$ it holds
\begin{equation}\label{EQN:PassCond1b}
\lim_{t \to \infty} p(t,x,B) = \rho(B)\,.
\end{equation}
and
\begin{equation}\label{EQN:PassCond}
\lim_{T \to \infty} \frac{1}{T}\int_0^T p(t,x,B)dt = \rho(B)\,.
\end{equation}
then that the SPHS is \textit{ultimately stochastic stable}.
\end{definition}

We will assume throughout the paper that there exists a feedback control law $u(t) = u(X(t))$ so that the SPHS admits a global solution. Also, without loss of generality we will assume $x_e = 0$. The next proposition states that if a SPHS is \textit{strictly ultimately stochastic passive}, then it is \textit{ultimately stochastic stable}.

\begin{proposition}\label{PRO:Conv}
Consider the SPHS \eqref{EQN:SPHS} and assume that:
\begin{description}[style=unboxed,leftmargin=0cm]
\item[(i)] it is \textit{strictly ultimately stochastic passive};
\item[(ii)] the noise is non--degenerate, that is the matrix $\Sigma(x) := \sigma(x) \sigma^T(x) \succ 0$ is positive definite.
\end{description}

Then, the SPHS \eqref{EQN:SPHS} admits a unique invariant measure and is \textit{ultimately stochastic stable}.
\end{proposition}
\begin{proof}
For the sake of readability we will divide the proof in several steps: in particular, in \textit{step 1} we will prove existence of an invariant measure, in \textit{step 2} we will prove the long time convergence of the transition density towards one of the invariant measures and at \textit{step 3} we will show uniqueness of the invariant measure.

\begin{description}[style=unboxed,leftmargin=0cm]
\item[(Step 1 - existence)] Under above assumptions the SPHS admits an invariant measure. In fact, since the SPHS is \textit{strictly ultimately stochastic passive}, there exists $\epsilon > 0$ such that, for $\| x \| \geq C$, $C > 0$, it holds
\begin{equation}\label{EQN:Est1}
\mathcal{L}H(x) < - \epsilon <0 \,.
\end{equation}

Then, using It\^{o} formula we have that
\begin{equation}\label{EQN:Est2}
\mathbb{E}H(X(t)) \leq H(x) + \int_0^t \mathbb{E} \mathcal{L}H (X(s)) ds\,.
\end{equation}

Using equations \eqref{EQN:Est1}--\eqref{EQN:Est2} we obtain
\begin{align}\label{EQN:Est3}
\mathbb{E} \mathcal{L}H (X(s)) &\leq C_m \mathbb{P}\left (\left .|X(s)| \leq C \right | X(0) = x\right ) - \epsilon \mathbb{P}\left (\left .|X(s)| > C \right | X(0) = x\right ) = \nonumber\\
&= - \epsilon + (C_m + \epsilon)\mathbb{P}\left (\left .|X(s)| \leq C \right | X(0) = x\right )\,,
\end{align}
being $C_m$ the maximum value of $\mathcal{L}H$ over the set $\|x\| \leq C$.

Using the fact that $H(x) \geq 0$, it follows from equation \eqref{EQN:Est2}
\[
- \frac{1}{t} \mathbb{E} H(X(0)) \leq \frac{1}{t}\int_0^t \mathbb{E} \mathcal{L}H (X(s)) ds\,,
\]
so that using estimate \eqref{EQN:Est3} we obtain
\begin{equation}\label{EQN:EstEInv}
- \frac{1}{t} \mathbb{E} H(X(0)) + \epsilon < (C_m+\epsilon) \frac{1}{t}\int_0^t \mathbb{P}\left ( |X(s)| \leq C \right ) ds\,.
\end{equation}

The existence of the invariant measure thus follows using equation \eqref{EQN:EstEInv} together with Proposition \ref{PRO:Las}.

\item[(Step 2 - convergence)] Let $\tau$ denote the first time the SPHS $X$ reaches the sphere $\|x\| \leq C$ and by $t \wedge \tau := \min\{t,\tau\}$; then, It\^{o} formula yields
\[
\mathbb{E}H(X(t \wedge \tau)) = H(x) + \int_0^{t \wedge \tau}\mathbb{E} \mathcal{L}H (X(s)) ds \leq H(x) -\epsilon \mathbb{E}[t \wedge \tau]\,.
\]

The fact that $H$ is non-negative implies
\[
\mathbb{E}[t \wedge \tau] \leq \frac{H(x)}{\epsilon}\,.
\]
In particular, assumption (B.2) in \cite[Assumption B, Chapter 4]{Kha} holds true. Therefore, using \cite[Theorem 4.2]{Kha} we obtain that, for a function $F$ integrable with respect to an invariant measure $\rho$, it holds
\begin{equation}\label{EQN:Conv1}
\mathbb{P}\left (\frac{1}{T}\int_0^T F(X(t))dt \to \int_{\RR^n} F(y)\rho(dy)\right )=1\,,\quad \mbox{as} \quad T \to \infty\,.
\end{equation}

If the function $F$ is bounded, equation \eqref{EQN:Conv1} implies using Lebesgue dominated convergence thereon,
\begin{equation}\label{EQN:Conv2}
\lim_{T \to \infty} \frac{1}{T}\int_0^T \mathbb{E} F(X(t))dt = \int_{\RR^n} F(x)\rho(dx)\,,
\end{equation}
which in turn implies, for $B \in \mathcal{B}(\RR^n)$, that
\begin{equation}\label{EQN:Conv3}
\lim_{T \to \infty} \frac{1}{T}\int_0^T p(t,x,B) dt = \rho(B)\,.
\end{equation}

At last, \cite[Theorem 4.3]{Kha} yields that
\begin{equation}\label{EQN:Conv4}
\lim_{t \to \infty} p(t,x,B) = \rho(B)\,.
\end{equation}

We have thus proven that the PSHS \eqref{EQN:SPHS} is \textit{ultimately stochastic stable}.

\item[(Step 3 - uniqueness)] To prove uniqueness of the invariant measure $\rho$, denote by $\rho_1$ another invariant measure. In particular, by the definition of invariant measure it holds
\begin{equation}\label{EQN:InvSec}
\int_{\RR^n}p(t,x,B) \rho_1(dx) = \rho_1(B)\,.
\end{equation}
Integrating equation \eqref{EQN:InvSec} in $[0,T]$ we obtain from equation \eqref{EQN:Conv3} that $\rho(B) = \rho_1(B)$ from which we infer the uniqueness of the invariant measure $\rho$. 
\end{description}
\end{proof}

Proposition \ref{PRO:Conv} establish a key result concerning the convergence of a \textit{strictly ultimately stochastic passive} towards the unique invariant measure. Such result is based, besides ultimately passivity of the SPHS, also on an assumption of non--degeneracy of the noise. Although such assumption can be considered to be fairly general in certain contexts, there is a general class of processes of particular interest that fails to satisfy the above non-degeneracy assumption. In particular, stochastic oscillator equations of the form
\begin{equation}\label{EQN:SOsc}
\begin{cases}
dq(t) &= f_q(p,q) dt\,,\\
dp(t) &= f_p(p,q)dt + \sigma(p,q)dW(t)\,,
\end{cases}
\end{equation}
do not satisfy non-degeneracy assumed in Proposition \ref{PRO:Conv}. Given the relevance of this class of systems, we will provide an alternative version of Proposition \ref{PRO:Conv} dropping the non-degeneracy assumption on the noise.

\begin{proposition}\label{PRO:ConvDeg}
Consider the SPHS \eqref{EQN:SPHS} and assume that:
\begin{description}
\item[(i)] it is \textit{strictly ultimately stochastic passive};
\item[(ii)] the transition kernel $p(t,x,B)$ is equivalent to the Lebesgue measure for any $t>0$ and $x \in \RR^n$. 
\end{description}
Then, the SPHS \eqref{EQN:SPHS} admits a unique invariant measure, which is absolutely continuous with respect to the Lebesgue measure, and it is \textit{ultimately stochastic stable}.
\end{proposition}
\begin{proof}
As in the proof of Proposition \ref{PRO:Conv} we will divide the current proof into three steps. Also, for the sake of brevity steps that follow from the same arguments as in Proposition \ref{PRO:Conv} will be skipped.

\begin{description}[style=unboxed,leftmargin=0cm]
\item[(Step 1 - existence)] Same arguments as in Step 1 of the proof of Proposition \ref{PRO:Conv} yield existence of an invariant measure.

\item[(Step 2 - convergence)] Since $p(t,x,A)$ is absolutely continuous with respect to the Lebesgue measure, denoting with a slight abuse of notation again by $p$ its density, we have that
\[
\int_{\RR^n} \int_B p(t,x,y)dy \rho(dx) = \rho(B)\,,
\]
so that Fubini theorem yields that also $\rho$ is absolutely continuous with respect to the Lebesgue measure. In the following we will denote again by $\rho$ the density of the invariant measure $\rho$. Thus, following \cite[Theorem 3]{Zak} we can infer equations \eqref{EQN:Conv1}--\eqref{EQN:Conv2}--\eqref{EQN:Conv3}--\eqref{EQN:Conv4}, so that the PSHS \eqref{EQN:SPHS} is \textit{ultimately stochastic stable}.

\item[(Step 3 - uniqueness)] Uniqueness of the invariant measure follows from equation \eqref{EQN:Conv3} as in the proof of Proposition \ref{PRO:Conv}.
\end{description}
\end{proof}

Therefore, to consider stochastic oscillator alike equation \eqref{EQN:SOsc}, the transition kernel of the driving process must be studied. The next example shows how Proposition \ref{PRO:ConvDeg} can be applied to a simple and yet relevant class of stochastic oscillator systems.

\begin{example}
Consider the system with additive noise
\begin{equation}\label{EQN:SOscAdd}
\begin{cases}
dq(t) &= p(t) dt\,,\\
dp(t) &= f_p(p,q)dt + \sigma dW(t)\,,
\end{cases}
\end{equation}
with $\sigma>0$. Since the transition probability kernels are equivalent under a change of probability measure, we can apply \textit{Girsanov theorem}, \cite{Kar}, and introduce $\tilde{W}$, a Brownian motion under the probability measure $\tilde{\mathbb{P}}$ equivalent to the original probability measure $\mathbb{P}$, defined as
\[
\tilde{W}(t) := W(t) + \sigma^{-1} \int_0^t f_p(p,q) ds \,.
\]  

Equation \eqref{EQN:SOscAdd} can be thus rewritten as
\begin{equation}\label{EQN:SOscAddEq}
\begin{cases}
dq(t) &= p(t) dt\,,\\
dp(t) &= \sigma d\tilde{W}(t)\,,
\end{cases}
\end{equation}
or in compact form as
\begin{equation}\label{EQN:OscComp}
dX(t)= AX(t)dt + \Sigma d\tilde{W}(t)\,,\quad X(t)=\left (q(t),p(t)\right )^T\,,
\end{equation}
for some suitable constant matrices $A$ and $\Sigma$. An integration with respect to time shows that the transition probability kernel is Gaussian distributed with covariance matrix given by
\begin{equation}\label{EQN:CovGauss}
\int_0^t e^{A(t-s)} \Sigma \Sigma^T e^{A^T(t-s)}ds\,.
\end{equation}

To show that $p$ is equivalent to the Lebesgue measure we must show that equation \eqref{EQN:CovGauss} is positive definite. This is equivalent to the requirement that the pair $(A,\Sigma)$ is controllable , that is the matrix $[\Sigma, A \Sigma,\dots,A^{n-1}\Sigma]$ span $\RR^n$. Therefore, a direct calculation implies that the transition probability kernel associated to \eqref{EQN:SOscAdd} is equivalent to the Lebesgue measure and therefore Proposition \ref{PRO:ConvDeg} applies.

This example can be generalized to other systems and the equivalence of the probability kernel can be checked via analogous arguments studying the controllability of the pair $(A,\Sigma)$.

\demo
\end{example}

\subsection{Energy balance of SPHS}\label{SEC:EB}

Consider the SPHS \eqref{EQN:SPHS}, then by mean of It\^{o}-formula, we have the following mean energy balance equation
\begin{equation}\label{EQN:EB}
\begin{split}
\mathbb{E}H(X(t)) - H(x) &= \mathbb{E}\int_0^t u^T(s) y(s) ds + \mathbb{E}\int_0^t \mathcal{L}H(X(s))ds = \\
&= \mathbb{E}\int_0^t u^T(s) y(s) ds - d(t),,
\end{split}
\end{equation}
where $d$ represents the dissipation of the system. The objective is thus to find a suitable control law $u = \phi(x) + \kappa$ such that the controlled SPHS has Hamiltonian function $H_d$ with minimum in a desired configuration $x_e$.

Consider the SPHS of the form
\begin{equation}\label{EQN:SPHSOr}
\begin{cases}
d X(t) = \left [(J-R) \partial_x H(X) + g(X(t)) u(t)\right ]dt + \sigma(X(t)) dW(t)\,,\\
y(t) = g^T(X(t)) \partial_x H(X(t))\,. 
\end{cases}
\end{equation}

We want to design a feedback control $u(t) = \phi(X(t))$ so that the controlled port--Hamiltonian system
\begin{equation}\label{EQN:SPHSTrans}
\begin{cases}
d X(t) = [J_d-R_d] \partial_x H_d(X(t)) dt + \sigma(X(t)) dW(t)\,,\\
y(t) = g^T(X(t)) \partial_x H_d(X(t))\,,
\end{cases}
\end{equation}
maintains the port--Hamiltonian structure.

The next result is the stochastic counterpart of the energy shaping result proved in \cite{OrtS} in the deterministic setting.

\begin{proposition}\label{PRO:ShapPC}
Given the SPHS \eqref{EQN:SPHSOr} and assume that it is possible to find $\phi$, $J_a$, $R_a$ and $K$ such that
\begin{equation}\label{EQN:Trans}
\left (\left [J+J_a\right ] - \left [R+R_a\right ]\right )K(X(t)) = g(X(t)) \phi(X(t)) - \left [J_a-R_a\right ]\partial_x H(X(t))\,,
\end{equation}
and such that the following conditions hold
\begin{description}[style=unboxed,leftmargin=0cm]
\item[(i) structure preservation:]
\begin{align}\label{EQN:SPb}
J_d &:= J+J_a = -\left (J+J_a\right )^T\,,\\
R_d &:= R+R_a = \left (R+R_a\right )^T \succeq 0\,;
\end{align}
\item[(ii) integrability:] it holds that
\begin{equation}\label{EQN:Int}
\partial_x K(X(t))=\partial^T_x K(X(t))\,;
\end{equation}
\item[(iii) equilibrium assignment] for a given $x_e \in \RR^n$, it holds
\begin{equation}\label{EQN:EqAss}
K(x_e) = - \partial_x H(x_e)\,;
\end{equation}
\item[(iv) stability] for $x_e \in \RR^n$, it holds
\begin{equation}\label{EQN:LStab}
\partial_x K(x_e) > - \partial_x^2 H(x_e)\,;
\end{equation}
\item[(v) ultimately stochastic passivity] for all $x$ such that $\|x-x_e\|>C$,
\begin{equation}\label{EQN:Lyap}
\mathcal{L}H_d(x) \leq - \epsilon <0 \,;
\end{equation}
\item[(vi) uniqueness of the invariant measure] either one of the following holds true:
\begin{description}[style=unboxed,leftmargin=0cm]
\item[(vi a) non-degeneracy] the noise is non--degenerate, that is the matrix $\Sigma(x) := \sigma(x) \sigma^T(x) \succ 0$ is positive definite;
\item[(vi b) equivalence] the transition kernel $p(t,x,B)$ is equivalent to the Lebesgue measure for any $t>0$ and $x \in \RR^n$.
\end{description}
\end{description}

Then the closed--loop system \eqref{EQN:SPHSOr} with control feedback law $u(t) = \phi(X(t))$ is a SPHS of the form
\begin{equation}\label{EQN:SPHSTransP}
\begin{cases}
d X(t) = [J_d-R_d] \partial_x H_d(X) dt + \sigma(X(t)) dW(t)\,,\\
y(t) = g^T(X(t)) \partial_x H(X(t))\,. 
\end{cases}
\end{equation}
with $H_d:=H+H_a$, $\partial_x H_a = K$ and there exists a unique invariant measure $\rho$ under which \eqref{EQN:SPHSTransP} is \textit{ultimately stochastic stable}.
\end{proposition}
\begin{proof}
Equations \eqref{EQN:Trans} ensures that its solution $K$ is such that the closed--loop SPHS is of the form \eqref{EQN:SPHSTransP} with total energy $H_d=H+H_a$.

Using Propositions \ref{PRO:Conv}--\ref{PRO:ConvDeg} it follows that the invariant measure is unique and the process is \textit{ultimately stochastic stable}.
\end{proof}

\begin{remark}\label{REM:Rem}
Several comments on Proposition \ref{PRO:ShapPC} are in order:
\begin{enumerate}[style=unboxed,leftmargin=0cm]
\item as regard condition $(v)$ on ultimately stochastic passivity in Proposition \ref{PRO:ShapPC}, it is worth noticing that the following holds true
\begin{equation}\label{EQN:CondSP}
\mathcal{L}H_d(x) = \mathcal{L}_0 H(x) + \mathcal{L}_a H_a(x) + \partial_x^T H_d(x)\left [J_d - R_d\right ]\partial_x H_d(x)\,,
\end{equation}
where $\mathcal{L}_0$ is the infinitesimal generator of the autonomous SPHS \eqref{EQN:SPHSOr} with null control $u\equiv 0$ and $\mathcal{L}_a$ is the infinitesimal generator of the autonomous SPHS \eqref{EQN:SPHSOr} with structure matrices $J_a$ and $R_a$. Using conditions \eqref{EQN:SPb}, equation \eqref{EQN:CondSP} reduces to
\begin{equation}\label{EQN:CondSP2}
\mathcal{L}H_d(x) \leq \mathcal{L}_0 H(x) + \mathcal{L}_a H_a(x) \,,
\end{equation}
so that if the original SPHS and the SPHS with structure matrices $J_a$ and $R_a$ are ultimately stochastic passives, then ultimately stochastic passivity holds true because
\[
\mathcal{L}H_d(x) \leq \mathcal{L}_0 H(x) + \mathcal{L}_a H_a(x) < -\epsilon\,.
\]

\item in \cite{Had} an alternative energy shaping was proposed. As already mentioned in the introduction, there is one key difference, with a fundamental implication, between Proposition \ref{PRO:ShapPC} and \cite[Theorem 4]{Had}. In \cite[Theorem 4]{Had} it is required that the noise vanishes at the equilibrium, that is $\sigma(x_e)=0$. Such condition has two relevant consequences: (i) an additive noise cannot be considered, and (ii) the Hamiltonian can be shaped only around the points in which the noise vanishes. Such conditions limit the range of application of \cite[Theorem 4]{Had}. Therefore, our proposed setting aims at filling the gap in which results in \cite{Had} cannot be applied. It is worth stressing nonetheless that, in the case of a vanishing noise we recover the same results;

\item invariant measures for general stochastic systems are not easy to explicitly derive. In fact, besides \textit{gradient systems}, in which case the invariant measure has an explicit exponential form, a general theory is of difficult derivation and each case must be studied ad hoc. Nonetheless, ergodicity of general stochastic systems is among the most studied properties of stochastic processes, so that a vast literature on this topic exists. In particular, sharp estimates on the support of the invariant measure can be obtained, \cite{Huang,Huang2}. Nonetheless, as already highlighted in \cite{OrtS}, the strengths of the proposed method is that it not necessary to explicitly compute the shaped Hamiltonian, whereas the main objective is to design a feedback-control construction procedure to stabilize a system. As it will be shown later with the aid of a specific example, a similar argument translates in the proposed stochastic setting in the sense that the explicit computation of the invariant measure of the system is not required;

\item the function $\phi$ in equation \eqref{EQN:Trans} can be obtained with analogous techniques as in the deterministic systems with null volatility $\sigma \equiv 0$. In particular, comments made in \cite[3.2]{OrtS} can be used to design the feedback control $\phi$. The only additional condition compared to a classical deterministic PHS is that the resulting SPHS \eqref{EQN:SPHSTransP} is \textit{ultimately stochastic passive}. This is among the main reasons for developing the weak energy shaping theory for SPHS.

More formally, consider a vanishing small noise $\sigma_\varepsilon:=\varepsilon \sigma$, either additive or multiplicative, and denote by $\rho_\varepsilon$ the invariant measure of the SPHS \eqref{EQN:SPHSTransP} with volatility $\sigma_\varepsilon$. Then, under some mild integrability assumptions on the regularity of the coefficients, is can be seen that, \cite{Huang}, $\rho_\varepsilon \to \rho$, $\varepsilon \to 0$, in the weak$^*$ topology of probability measure, that is
\[
\int_{\RR^n}f(x) \rho_\varepsilon(dx) \to \int_{\RR^n}f(x) \rho(dx) \,,\quad \mbox{as} \quad \varepsilon \to 0\,,
\]
for any $f \in C_b(\RR^n)$. Also, the limiting measure $\rho$, corresponding to the deterministic PHS \eqref{EQN:SPHSTransP} with null volatility $\sigma \equiv 0$, can be seen to have support in
\[
supp(\rho) \subset S= \left \{x \in \RR^n \,:\,  \partial_x^T H_d(x)  R_d(x) \partial_x H_d(x) = 0 \right \}\,.
\]

Therefore in the case $S=\{x_e\}$, the deterministic PHS is asymptotically stable and $\rho_\varepsilon \to \delta_{x_e}$ as $\varepsilon \to 0$, being $\delta_{x_e}$ the Dirac measure concentrated in $x_e$, \cite{Huang}. This clarifies the structure of Proposition \ref{PRO:ShapPC} in the sense that the shaped SPHS possesses an invariant measure which is shaped around the equilibrium that is typically specified in the deterministic context. In this sense, since in general the noise is an external disturbance that cannot be removed, the propose weak energy shaping aims at stabilizing a system around a desired configuration given a certain environmental noise that cannot be removed or compensated. 
\end{enumerate}
\demo
\end{remark}

Proposition \ref{PRO:ShapPC} clarifies the main idea behind the proposed setting. In general, a stochastic system is a dynamic system subject to external random perturbations. Since in the proposed setting the noise cannot be affected by the control law, the main idea is to shape the dynamic system around the equilibrium desired for the deterministic system. This shaping is done considering ergodic properties of the stochastic systems. In particular, to a smaller noise corresponds to an invariant measure concentrated around the equilibrium of the deterministic PHS. A further relevant point opened by the current research is the opposite point of view. That is, given a deterministic PHS, a controller can inject a suitable noise into the system such that the resulting system enjoys better stability properties, such as faster convergence to the equilibrium, \cite{Arn}. For instance, this can be achieved via a suitable multiplicative noise of increasing magnitude in certain domains. It is worth stressing that results show that, suitably injecting a random perturbation into a stochastic system can stabilize a deterministic dynamic system that would not be stabilizable otherwise. Such stabilization-by-noise is not treated in the current research and will be the topic of future research.

\section{On the connection with infinite dimensional deterministic Port--Hamiltonian systems}\label{SEC:ConnInf}

The present Section aims at showing a suggestive connection between SPHS and infinite--dimensional deterministic PHS. As introduced in previous Sections, the weak energy shaping approach is based on the invariant measure of an SPHS. In order to study the invariant measure of a SDE, a common and powerful approach is to study the \textit{Fokker--Planck} equation associated to the SDE. We will show that the related \textit{Fokker-Planck} equation is an infinite--dimensional PHS on a suitable functional space, so that in turn the problem of weak energy shaping of stochastic PHS can be associated to the problem of energy shaping of an infinite--dimensional controlled PHS.

Consider the SPHS
\begin{equation}\label{EQN:SPHSID2}
\begin{cases}
dX(t) = \left [(J-R)\partial_x H(X(t)) + g(X(t)) u(t)\right ]dt + \sigma(X(t)) dW(t)\,,\\
y(t)=g^T(X(t)) \partial H(X(t))\,.
\end{cases}
\end{equation}

As introduced in Section \ref{SEC:SPSDE}, the transition semigroup $P_t$ defined as
\[
P_t f(x) := \mathbb{E} f(X^x(t))\,,
\] 
for any $f \in C_b(\RR^n)$. 

Under certain regularity condition, \cite{Lor}, the function $v(t,x)$ defined through \textit{Markov semigroup} in equation \eqref{EQN:MarkovSp} as
\[
v(t,x) := P_t f(x)\,,
\]
is the solution of the Cauchy problem
\begin{equation}\label{EQN:CauchyID2}
\begin{cases}
\partial_t v(t,x) - \mathcal{L}v(t,x) = 0\,,\quad (t,x) \in \RR_+ \times \RR^n\,,\\
v(0,x) = f(x)\,,
\end{cases}
\end{equation}
where $\mathcal{L}$ is the infinitesimal generator of the Markov process defined as
\begin{align}\label{EQN:InfGSPHS2}
\mathcal{L}f(x) &= \sum_{i=1}^n \mu_i(x) \partial_{x_i}f(x) + \frac{1}{2}\sum_{i,j=1}^n \Sigma_{ij}(x)\partial^2_{x_i \, x_j}f(x)\,,
\end{align}
where we set for short
\[
\mu:=  \left [(J-R)\partial_x H(X(t)) + g(X(t)) u(t)\right ]\,,
\]
and $\Sigma(x) := \sigma^T(x)\sigma(x)$. The formal adjoint $\mathcal{L}^*$ in distributional sense, \cite{Kry}, of the \textit{infinitesimal generator} $\mathcal{L}$ is given by
\begin{align}\label{EQN:InfGSPHS2Adj}
\mathcal{L}^*f(x) &= -\sum_{i=1}^n  \partial_{x_i}\left (\mu_i(x)f(x)\right ) + \frac{1}{2}\sum_{i,j=1}^n \partial^2_{x_i \, x_j}\left (\Sigma_{ij}(x)f(x)\right )\,.
\end{align}

In the following we will denote by $L^p_\rho := L^p(\RR^n,\rho)$, $p\in [1,\infty)$, the space of $p-$integrable functions with respect to the measure $\rho$; $\| \cdot \|_p$ is the standard norm in the space $L^p_\rho$.

Definition \ref{DEF:Inv} can be restated in terms of the Markov semigroup $P_t$ as
\begin{equation}\label{EQN:InvMarkovP22}
\int_{\RR^n} P_t f(x)\rho(dx) = \int_{\RR^n} f(x) \rho(dx)\,,\quad f \in B_b(\RR^n)\,.
\end{equation}

A key aspect is that, if $P_t$ admits a unique invariant measure according to equation \eqref{EQN:InvMarkovP22}, then the
semigroup can be extended to a semigroup of bounded operators in the space $L^p$, \cite[Chapter 9]{Lor}. If no confusion is possible, typically such extension is still denoted by $P_t$. Among the most relevant aspects of the extended semigroup is that it is possible to study into details the long--time behaviour of the semigroup. In particular, it can be proved that the function $P_t f$ converges to 
\[
\bar{f}_\rho := \int_{\RR^n} f(x) \rho(dx) \quad \mbox{in} \quad L^p \quad \mbox{as} \quad t \to \infty\,.
\]
Therefore, if the semigroup $P_t$ is the solution of an infinite--dimensional PHS, as we will show in a while, the problem of weak energy shaping for the SPHS \eqref{EQN:SPHSID2} can be associated to the energy shaping of the associated infinite--dimensional PHS.

We will not enter into details regarding neither the properties of the Markov semigroup $P_t$ or its extensions to a semigroup on $L^p_\rho$ since it is a topic extensive treated in literature. We refer the reader to \cite{Lor} for further details.

Assume that Proposition \ref{PRO:ShapPC} holds and that $\sigma$ is non--degenerate. Therefore, the SPHS \eqref{EQN:SPHSID2} can be rewritten as
\begin{equation}\label{EQN:SPHSTransPID2}
\begin{cases}
d X(t) = [J_d-R_d] \partial_x H_d(X) dt + \sigma(X(t)) dW(t)\,,\\
y(t) = g^T(X(t)) \partial_x H_d(X(t))\,. 
\end{cases}
\end{equation}

In particular, the autonomous SPHS \eqref{EQN:SPHSTransPID2} admits a unique invariant measure $\rho$. Therefore, following \cite{Lor}, we can show that the Markov semigroup $P_t$ associated to the autonomous SPHS \eqref{EQN:SPHSTransPID2} extends to a strongly continuous semigroup on the space $L^p_\rho$, $p \in [1,\infty)$.

In the following, we assume that the unique invariant measure admits a density and it is invertible. We will consider the Hilbert space setting so that we set $p=2$. We recall that, being $L^2_\rho$ a Hilbert space, it can be endowed with a natural inner product defined as
\[
\langle f,g \rangle_\rho := \int f(x)g(x)\rho(x)dx\,.
\]

In such a case we have the following.

\begin{proposition}\label{PRO:PHSDec}
The \textit{infinitesimal generator} $\mathcal{L}$ in equation \eqref{EQN:InfGSPHS2} can be decomposed into a symmetric and anti--symmetric operator, i.e.
\begin{equation}\label{EQN:DecL}
\mathcal{L} = \mathcal{L}_{as} + \mathcal{L}_s\,,
\end{equation}
with $\mathcal{L}_{as}$, resp. $\mathcal{L}_s$, an anti--symmetric, resp. symmetric, operator on $L^2(\RR^n,\rho(x)dx)$.
\end{proposition}
\begin{proof}
Notice first that, integrating by parts, the following holds true,
\begin{align}\label{EQN:SymmMa}
\begin{split}
&\langle g,\mathcal{L}f \rangle_\rho =\\
&= \sum_{i=1}^n \int g(x) \mu_i(x) \partial_{x_i}f(x) \rho(x)dx + \frac{1}{2}\sum_{i,j=1}^n \int g(x) \Sigma_{ij}(x)\partial^2_{x_i \, x_j}f(x) = \\
&= -\sum_{i=1}^n \int \partial_{x_i} \left (g(x) \mu_i(x) \rho(x)\right ) f(x)dx + \frac{1}{2}\sum_{i,j=1}^n \int \partial^2_{x_i \, x_j}\left (g(x) \Sigma_{ij}(x)\rho(x)\right )f(x) dx=\\
&=\langle \rho^{-1} \mathcal{L}^*(\rho g),f \rangle_\rho\,.
\end{split}
\end{align}

We can thus define

\begin{equation}\label{EQN:AtS}
\begin{cases}
\mathcal{L}_{s}f &:= \frac{1}{2} \left (\mathcal{L}f + \rho^{-1} \mathcal{L}^*(\rho f)\right )\,,\\
\mathcal{L}_{as}f &:= \frac{1}{2} \left (\mathcal{L}f - \rho^{-1} \mathcal{L}^*(\rho f)\right )\,.
\end{cases}
\end{equation}

It is immediate to see that
\[
\mathcal{L} = \mathcal{L}_{as} + \mathcal{L}_s\,.
\]

It further holds that, using equation \eqref{EQN:SymmMa},
\begin{align*}
\langle g,\mathcal{L}_s f \rangle_\rho &= \frac{1}{2} \langle g,\mathcal{L}f \rangle_\rho+ \frac{1}{2} \langle g,\rho^{-1} \mathcal{L}^*(\rho f) \rangle_\rho = \\
&=\frac{1}{2} \langle \rho^{-1} \mathcal{L}^*(\rho g),f \rangle_\rho+ \frac{1}{2} \langle \mathcal{L}g,f \rangle_\rho =\langle \mathcal{L}_s g, f \rangle_\rho\,,
\end{align*}
and
\begin{align*}
\langle g,\mathcal{L}_{as} f \rangle_\rho &= \frac{1}{2} \langle g,\mathcal{L}f \rangle_\rho - \frac{1}{2} \langle g,\rho^{-1} \mathcal{L}^*(\rho f) \rangle_\rho = \\
&=\frac{1}{2} \langle \rho^{-1} \mathcal{L}^*(\rho g),f \rangle_\rho - \frac{1}{2} \langle \mathcal{L}g,f \rangle_\rho =-\langle \mathcal{L}_{as} g, f \rangle_\rho\,,
\end{align*}
so that the operator $\mathcal{L}_{as}$, resp. $\mathcal{L}_{s}$, is anti--symmetric, resp. symmetric.
\end{proof}

Therefore, we can establish the following formulation of the Fokker-Planck equation \eqref{EQN:CauchyID2} in term of a Hilbert--space valued deterministic PHS, \cite{LeG,Jac}.

\begin{proposition}
The Cauchy problem in equation \eqref{EQN:CauchyID2} defines an infinite--dimensional PHS on the Hilbert space $L^2_\rho$ with linear Hamiltonian. In particular, the probability density is conserved.
\end{proposition}
\begin{proof}
Using Proposition \ref{PRO:PHSDec} it can be seen that the Cauchy problem \eqref{EQN:CauchyID2} can be written in terms of a symmetric and anti--symmetric operator as
\begin{equation}\label{EQN:CauchyID2SAS}
\begin{cases}
\partial_t v(t,x) - \left (\mathcal{L}_{as} + \mathcal{L}_s\right )v(t,x) = 0\,,\quad (t,x) \in \RR_+ \times \RR^n\,,\\
v(0,x) = f(x)\,.
\end{cases}
\end{equation}

Equation \eqref{EQN:CauchyID2SAS} defines therefore a linear PHS on the space $L^2_\rho$ with linear Hamiltonian $\mathcal{H}v = v$. Using \cite[Proposition 9.1.9]{Lor} we have that the semigroup $P_t$ is conservative on the space $L^2_\rho$, that is
\[
\int_{\RR^n} P_t \mathbf{1}\rho(x)dx = \int_{\RR^n}\mathbf{1}\rho(x)dx=1\,.
\]

From the fact that $\rho$ is an invariant measure we immediately have that
\[
\int_{\RR^n} \mathcal{L}f(x) \rho(x)dx = 0\,.
\]
Using \cite[Theorem 1.3.4]{Lor}, it follows
\[
\frac{d}{dt} P_t f(x)=\mathcal{L} P_t f(x)\,.
\]
which yields, after integration in $\RR^n$, that
\[
\frac{d}{dt} \int_{\RR^n} P_t f(x) \rho(x)dx = \int_{\RR^n} \mathcal{L} P_t f(x) \rho(x)dx = 0\,.
\]

At last, using the fact that
\[
v(t,x) := P_t f(x)\,,
\]
we obtain
\[
\frac{d}{dt} \int_{\RR^n} v(t,x) \rho(x)dx = 0\,,
\]
which states that the probability density function is conserved.
\end{proof}

Finally, we can prove a convergence result.

\begin{proposition}
For any $f \in L^p_\rho$, $p\in [1,\infty)$, we have
\[
\lim_{t \to \infty} P_t f = \bar{f}:= \int_{\RR^n}f(x)\rho(x)dx\quad \mbox{in} \quad L^2_\rho \,.
\]
\end{proposition}
\begin{proof}
It follows from \cite[Theorem 9.1.16]{Lor}.
\end{proof}

The above result shows a key connection between the weak stochastic energy shaping proposed in the present paper and the infinite dimensional shaping of deterministic Hilbert space valued PHS. The current work addresses the implication that a weak stochastic shaping has on an associated infinite dimensional PHS. Nonetheless in several application the opposite direction is of relevant interest, where deterministic infinite dimensional technique can be used to properly shape a stochastic PHS. Particular interest in this direction is played by the the H-theorem used in physical systems, \cite{Bar}.

\section{Energy shaping for stochastic port--Hamiltonian systems}\label{SEC:EnSh}

The final goal of energy shaping techniques is to change the shape of the Hamiltonian function so that it has a minimum at a desired point $x_e$. 
%As emerged above, since in the stochastic setting the control enters only in the drift, it is not possible to shape the diffusion coefficient to have a minimum at the same desired value $X_e$. This implies that either the diffusion already vanishes at $X_e$ or the stochastic system cannot be stabilized in standard probabilistic sense. To overcome this problem we weaken the notion of passivity and stability so that we are not aiming at stabilizing the system at $X_e$ but we wish to shape the Hamiltonian function so that the resulting process will stay bounded in a bounded domain around the desired equilibrium $X_e$. This will emerge in the particular form for the invariant measure of the controlled system.
Under the above described point of view, energy shaping for SPHS assumes a purely probabilistic interpretation. In fact, the final goal will be to properly shape the invariant measure of the stochastic system so that the process will evolve al large time according to a desired invariant law, peaked around the desired equilibrium configuration $x_e$.

Before entering into details let state some results regarding \textit{Casimir function} for SPHS.

\subsection{On Casimir for stochastic port-Hamiltonian systems}\label{SEC:Cas}

In the present section we are to extend the notion of \textit{Casimir} to stochastic PHS. As usual in the stochastic contest, different notions of conserved quantities can be introduced, usually referred to as \textit{strong} or \textit{weak} conserved quantities, see, e.g. \cite{Ort,CDPMSPHS}.

Consider the \textit{stochastic PHS} \eqref{EQN:SPHS}. We will use the following definitions of strong and weak \textit{Casimir} function.

\begin{definition}\label{DEF:Casimir}
\begin{description}[style=unboxed,leftmargin=0cm]
\item[(i)] A function $C: \RR^n \to \RR$ is called a \textit{strong Casimir} for the SPHS \eqref{EQN:SPHS} if
\[
dC(X(t)) = 0\,.
\]

\item[(ii)] A function $C: \RR^n \to \RR$ is called a \textit{weak Casimir} for the SPHS \eqref{EQN:SPHS} if
\[
\mathcal{L} C(X(t)) = 0\,.
\]
\end{description}
\end{definition}

It is worth mentioning that the notion of \textit{strong Casimir} is the straightforward generalization of deterministic Casimir, in the sense that a Casimir is a quantity that is conserved along the trajectory of the system. Nonetheless, as already emerged in \cite{CDPMSPHS}, it is often too strong to require that a quantity is conserved $\mathbb{P}-$a.s. so that in practical applications it is usually preferable to consider a weak conservation. Such choice of weak conserved quantities will have strong implication in the following treatment.

Notice that an application of Dynkin formula yields that, if $C$ is a \textit{weak Casimir}, then it holds
\begin{equation}\label{EQN:DynC}
\mathbb{E}C(X(t)) = C(x) + \mathbb{E}\int_0^t \mathcal{L}C(X(s))ds =  C(x) = c\,.
\end{equation}

We thus have the following.

\begin{theorem}\label{THM:Casimir}
Given the autonomous stochastic PHS \eqref{EQN:SPHS} with constant null control $u \equiv 0$. 
\begin{description}[style=unboxed,leftmargin=0cm]
\item[(i)] If
\begin{equation}\label{EQN:StrongC}
\begin{cases}
\partial^T_x C(x)\left [ J(x) - R(x)\right ]=0\,,\\
\frac{1}{2} Tr\left [\sigma^T(x) \partial^2_x C(x) \sigma(x)\right ]=0\,,\\
\partial^T_x C(x) \sigma(x)=0\,,
\end{cases}
\end{equation}
then $C$ is a \textit{strong Casimir} for the SPHS \eqref{EQN:SPHS}.

\item[(ii)] If
\begin{equation}\label{EQN:WeakC}
\begin{cases}
\partial^T_x C(x)\left [ J(x) - R(x)\right ]=0\,,\\
\frac{1}{2} Tr\left [\sigma^T(x) \partial^2_x C(x) \sigma(x)\right ]=0\,,
\end{cases}
\end{equation}
then $C$ is a \textit{weak Casimir} for the SPHS \eqref{EQN:SPHS}.
\end{description}
\end{theorem}
\begin{proof}
\begin{description}[style=unboxed,leftmargin=0cm]
\item[(i)] It\^{o} formula applied to a function $C : \RR^n \to \RR$, yields
\begin{align}
&dC(X(t)) = \partial^T_x C(X(t)) dX(t) + \frac{1}{2} Tr\left [\sigma^T(X(t)) \partial^2 C(X(t)) \sigma(X(t))\right ] dt =\\
&= \left (\partial^T_x C(X(t))\left [ J(x) - R(x)\right ] \partial_x H(X(t)) + \frac{1}{2} Tr\left [\sigma^T(X(t)) \partial^2 C(X(t)) \sigma(X(t))\right ]\right ) dt + \\
&\qquad +\partial^T_x C(X(t)) \sigma(X(t))dW(t)\,.
\end{align}

Using therefore conditions \eqref{EQN:StrongC} it follows that $dC(X(t)) = 0$ and according to Definition \ref{DEF:Casimir} the claim follows.

\item[(ii)] Considering the infinitesimal generator $\mathcal{L}$, taking the expected values and using the martingale property of the It\^{o} integral we obtain
\[
\mathcal{L}C(x) = \partial^T_x C(x)\left [ J(x)-R(x)\right ] \partial_x H(x) + \frac{1}{2} Tr\left [\sigma^T(x) \partial^2_x C(x) \sigma(x) \right ]\,,
\]
so that, using conditions \eqref{EQN:WeakC}, it follows $\mathcal{L}C(X(t))=0$ and according to Definition \ref{DEF:Casimir} we have the claim.
\end{description}
\end{proof}

%
%Let us assume that the the noise is given in potential form and carries the \textit{Poisson structure}, that is 
%\begin{equation}\label{EQN:SigmaPot}
%\sigma(X(t)) = J(X(t)) \partial \gamma(X(t))\,,
%\end{equation}
%for some suitable function $\gamma$. Notice that above requirement is a standard requirement in \textit{stochastic geometric mechanics} so that PHS \eqref{EQN:SPHS} can be written for any function $\varphi$ in terms of Poisson bracket as
%\[
%d\varpi(X(t))= \{\varphi,H\}(X(t)) dt + \{\varphi,\gamma\}(X(t))\circ dW(t)\,.
%\]
%
%Then Theorem \ref{THM:Casimir} can be rephrased as
%\begin{theorem}\label{THM:CasimirII}
%Let us consider the unforced stochastic PHS \eqref{EQN:SPHS}, or its It\^{o} equivalent \eqref{EQN:SPHSIto}, with no dissipation, that is set $R=0$ and consider the constant null control $u \equiv 0$. Let further assume that $\sigma$ is of the form given in equation \eqref{EQN:SigmaPot}. Then if
%\begin{equation}\label{EQN:StrongC}
%\begin{cases}
%\partial^T C(X(t)) J(X(t))=0\,,\\
%\end{cases}
%\end{equation}
%then $C$ is a \textit{strong Casimir} for the SPHS \eqref{EQN:SPHS}.
%
%If instead
%\begin{equation}\label{EQN:WeakC}
%\begin{cases}
%\partial^T C(X(t)) J(X(t))=0\,,\\
%\frac{1}{2} \partial^T C(X(t)) \bar{\sigma}(X(t))=0\,,\\
%\frac{1}{2} Tr\left [\sigma^T(X(t)) \partial^2 C(X(t)) \sigma(X(t))\right ]=0\,,
%\end{cases}
%\end{equation}
%then $C$ is a \textit{weak Casimir} for the SPHS \eqref{EQN:SPHS}.
%\end{theorem}
%\begin{proof}
%
%\end{proof}

\subsection{Control as interconnection}

The present section is devoted to energy--based control of SPHS; the main goal is to design a feedback interconnection of SPHS such that the closed--loop system is stable in a suitable sense.

A $\RR^{n_C}-$valued controller in SPHS form is given by
\begin{equation}\label{EQN:SPHSCon1}
\begin{cases}
d Z(t) = \left [\left (J_c(Z(t))-R_c(Z(t))\right )\partial_z H_c(z) + g_c(Z(t)) u_c \right ] + \sigma_c(Z(t)) dB(t)\,,\\
y_c(t) = g_c^T(Z(t)) \partial_z H_c(Z(t))\,,
\end{cases}
\end{equation}
where $B$ is a standard Brownian motion independent of $W$.

We interconnect the controller \eqref{EQN:SPHSCon1} to the original SPHS \eqref{EQN:SPHS} through the power preserving interconnection
\[
u = -y_c + v\,,\quad u_c = y + v_c\,.
\]

\begin{figure}[!t]
\centering     %%% not \center
\includegraphics[width=.8\columnwidth]{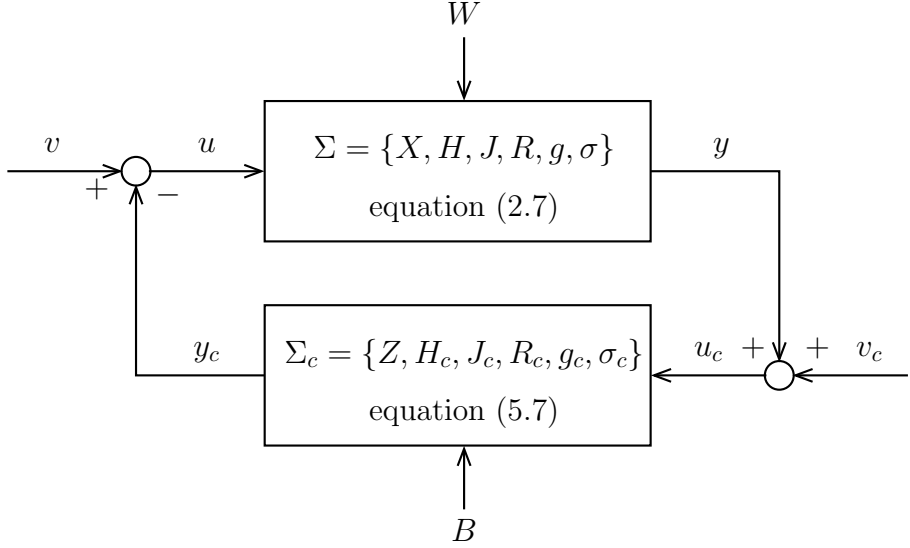}
\caption{Block diagram of the control by interconnection scheme with external ports $(v,y)$ and external noise $(B,W)$.}\label{FIG:TL}
\end{figure}

Therefore it can be seen that, see, e.g. \cite{CDPMSPHS}, the interconnected system in Figure \ref{FIG:TL} is still a SPHS of the form
\begin{equation}\label{EQN:SPHSInt1}
\begin{cases}
d 
\begin{pmatrix}
X(t)\\
Z(t)
\end{pmatrix}
&= \left (
\begin{pmatrix}
J & - g g^T_c\\
g_c g^T & J_c
\end{pmatrix}
-
\begin{pmatrix}
R & 0\\
0 & R_c
\end{pmatrix}
\right ) 
\begin{pmatrix}
\partial_x H(X(t))\\
\partial_z H_c(Z(t))
\end{pmatrix}
dt +\\
&+
\begin{pmatrix}
g & 0\\
0 & g_c
\end{pmatrix}
\begin{pmatrix}
v\\
v_c
\end{pmatrix}
dt + 
\begin{pmatrix}
\sigma(X(t)) dW(t)\\
\sigma_c(Z(t)) dB(t)
\end{pmatrix}\\
\begin{pmatrix}
y(t)\\
y_c(t)
\end{pmatrix}
&= 
\begin{pmatrix}
g^T & 0\\
0 & g_c^T
\end{pmatrix}
\begin{pmatrix}
\partial_x H(X(t))\\
\partial_z H_c(Z(t))
\end{pmatrix}\,.
\end{cases}
\end{equation}

\begin{remark}
We stress that, if not otherwise specified, in the present work we will always consider weak Casimir. The choice is motivated by many reasons. Firstly, as already mentioned above, existence of strong Casimir is often unrealistic and simple examples can be found where no strong Casimir exists. Secondly, when the attention is turned to energy shaping, the choice of strong Casimir poses even greater problems. In \cite{Had} for instance energy shaping via strong Casimir is studied and in \cite[Proposition 5]{Had} sufficient conditions are derived in order to ensure that a given function is a strong Casimir. In particular, the authors must assumes that the controller in SPHS form is perturbed by the same noise as the original SPHS. This is due to the fact that, since a quantity must be conserved along the trajectories, the controller must compensate $\mathbb{P}-$a.s. the noise due to the system. Such condition require further a complete knowledge of the noise in the sense that in a real application the controller must be able to discern the contribution due to the noise from the real state of the system. Such assumptions are difficult to be satisfied in practice.

The choice of a weak Casimir on the contrary overcomes such issues. In fact, since a quantity, as shown in equation \eqref{EQN:DynC}, is required to be conserved in mean value, it is enough to design a controller that matches the SPHS on average. Therefore a filtering of the state of the system can be used to disentangle the contribution of the noise from the state of the system. These considerations will emerge in later results.

\demo
\end{remark}

We thus look for (weak) \textit{Casimir functions} $C(x,z)$ of the form
\begin{equation}\label{EQN:CasCI}
C_i(x,z) = F_i(x) - S_i(z_i)\,,\quad i =1,\dots,n_c\,,
\end{equation}
for some regular enough functions $F_i:\RR^n \to \RR$ and $S_i:\RR^{n_c}\to \RR$.

We thus have the following.

\begin{proposition}\label{PRO:CondCas2}
If, for all $i =1,\dots,n_c$, it holds
\begin{equation}\label{EQN:InfConec}
\begin{split}
&\begin{pmatrix}
\partial^T_x F_i (x)\left (J-R \right ) - \partial^T_z S_i (z) g_c(z)g(x) \\
\partial^T_x F_i(x)g(x)g^T_c(z) -\partial^T_z S_i (z)\left (J_c-R_c\right )
\end{pmatrix}^T
\begin{pmatrix}
\partial_x H(x)\\
\partial_z H_c(z)
\end{pmatrix} +\\
&+ \frac{1}{2} Tr\left [\sigma^T(x) \partial^2_x F_i(x) \sigma(x) - \sigma^T_c(z)\partial^2_z S_i (z)\sigma_c(z)\right ] = 0\,,
\end{split}
\end{equation}
then the functions
\[
C_i(x,z) = F_i(x) - S_i(z_i)\,,\quad i =1,\dots,n_c\,,
\]
are \textit{weak Casimirs} for the closed--loop SPHS \eqref{EQN:SPHSInt1}.
\end{proposition} 
\begin{proof}
An application of It\^{o} formula yields that
\begin{equation}\label{EQN:InfConec2}
\begin{split}
&dC_i(X(t),Z(t)) =\\ 
&=\begin{pmatrix}
\partial^T_x F_i (X(t))\left (J-R \right ) - \partial^T_z S_i (z)g_c(Z(t))g(X(t)) \\
\partial^T_x F_i(X(t))g(X(t))g^T_c(Z(t)) -\partial^T_z S_i (z)\left (J_c-R_c\right )
\end{pmatrix}^T
\begin{pmatrix}
\partial_x H(X(t))\\
\partial_z H_c(Z(t))
\end{pmatrix}dt +\\
&+ \frac{1}{2} Tr\left [\sigma^T(X(t)) \partial^2_x F_i(X(t)) \sigma(X(t)) - \sigma^T_c(z)\partial^2_z S_i (z)\sigma_c(z)\right ] dt+\\
&+ \partial^T_x F_i(X(t))\sigma(X(t))dW(t) + \partial^T_z S_i (z) \sigma_c(Z(t))dB(t) \,.
\end{split}
\end{equation}
Taking the expectation implies that if equation \eqref{EQN:InfConec} holds then $C_i(x,z)$ is a \textit{weak Casimir} according to Definition \ref{DEF:Casimir}.
\end{proof}

As regard Proposition \ref{PRO:CondCas2} we have the following sufficient conditions for the condition \eqref{EQN:InfConec}.

\begin{proposition}\label{PRO:CondCas2SC}
If there exist functions $F_i$ and $S_i$ such that
\begin{align}\label{EQN:CondCasSC}
\begin{cases}
\partial_x^T F_i(x) J(x) \partial_x F_i(x) = \partial^T_z S_i(z)J_c(z)\partial_z S_i(z)\,,\\
R(x) \partial_x F_i(x) = 0\,,\\
\partial_z S_i(z)R_c(z)=0\,,\\
\partial_x^T F_i(x) J(x) = \partial_z S_i(z)g_c(z) g^T(x)\,,\\
Tr\left [\sigma^T(x)\partial_{x}^2 F_i(x) \sigma(x)- \sigma^T_c(z)\partial^2_z S_i (z)\sigma_c(z)\right ]=0\,.
\end{cases}
\end{align}
then equation \eqref{EQN:InfConec} holds and the functions
\[
C_i(x,z) = F_i(x) - S_i(z)\,,\quad i =1,\dots,n_c\,,
\]
is a \textit{weak Casimir} for the closed--loop SPHS \eqref{EQN:SPHSInt1}.
\end{proposition} 

\begin{remark}
Sufficient conditions \eqref{EQN:CondCasSC} highlight some key aspects of the proposed approach:
\begin{enumerate}
\item as in the deterministic case, the second condition
\[
R(x) \partial_x F_i(x) = 0\,,
\]
implies that the direction in which dissipation happens cannot be shaped. This limitation is known in literature as \textit{dissipation obstacle}, \cite{Fan,OrtS}. Nonetheless, conditions \eqref{EQN:CondCasSC} are only sufficient so that future research will be devoted to understand if the setting developed can be used to overcome the \textit{dissipation obstacle};

\item the last condition states that the controller's noise  must compensate, on average, the noise of the system. This implies two immediate comments: (i) differently from \cite{Had} it is only necessary that the controller's noise  compensates the noise of the system on average instead that $\mathbb{P}-$a.s.: such condition is easy to be satisfied in practice, and (ii) the controller in equation \eqref{EQN:SPHSCon1} has been designed with a noise term so that it can be chosen to match the contribution due to the system noise.
\end{enumerate} 

\demo
\end{remark}
\begin{proof}
The proof follows straightforward by checking that if conditions in \eqref{EQN:CondCasSC} are valid then \eqref{EQN:InfConec} follows.
\end{proof}

The infinitesimal generator for the closed--loop dynamics of $X$ in equation \eqref{EQN:SPHSInt1} now becomes, for any $\varphi$,
\begin{align}
\mathcal{L} \varphi(x) &= \partial_x^T \varphi(x) \left ( (J(x) -R(x)) \partial_x H(x) - g(x) g^T_c(z) \partial_z H_c(z) \right )\nonumber \\
& + \frac{1}{2} Tr \left [\sigma^T(x) \partial_{x}^2 \varphi(x) \sigma(x) \right ]\,.
\end{align}

We can thus restrict the dynamics on the set
\[
\{(x,z) \,:\, F_i(x) + c_i = S_i(z)\,,\,i=1,\dots,n_C\}\,,
\]
where $c_i:= C_i(x)$ is the constant value assumed by the Casimir $C_i$ according to Dynkin formula \eqref{EQN:DynC}. Thus, since $C$ is a Casimir and using second and third conditions in equation \eqref{EQN:CondCasSC} we have that
\begin{align}\label{EQN:InfG21}
\mathcal{L} \varphi(x) &= \partial_x^T \varphi(x) (J(x) -R(x))\left ( \partial_x H(x) + \partial_z H_C(z)\partial_x F(x) \right ) +\\
&+ \frac{1}{2} Tr \left [\sigma^T(x) \partial_{xx}^2 \varphi(x) \sigma(x) \right ]\,.
\end{align}

Thus, on $z_i = F_i(x) + c_i$, setting 
\[
H_d(x) := H(x) + H_c(c + F(x))\,,
\]
equation \eqref{EQN:InfG21} can be rewritten as
\begin{align}\label{EQN:InfGFin}
\mathcal{L} \varphi(x) &= \partial_x^T \varphi(x) (J(x) -R(x)) \partial_x H_d(x) + \frac{1}{2} Tr \left [\sigma^T(x) \partial_{xx}^2 \varphi(x) \sigma(x) \right ]\,,
\end{align}
so that it is the infinitesimal generator of the stochastic PHS
\begin{equation}\label{EQN:NLSPHS2}
d X(t) = (J(x) -R(x)) \partial_x H_d(X(t))dt + \sigma(X(t)) dW(t)\,.
\end{equation}

%In particular we have,
%\[
%\mathcal{L}H_N(x) = \mathcal{L}H(x) + \mathcal{L}H_C(z) \leq \mathcal{L}H(x) + Tr\left [\sigma^T_C(z) \partial_{zz}^2 H_C(z) \sigma_C(z)\right ] + y^T_C u_C\,,
%\]
%and using the power preserving interconnection we obtain
%\[
%\mathcal{L}H_N(x) = \mathcal{L}H(x) + \mathcal{L}H_C(z) \leq \mathcal{L}H(x) + Tr\left [\sigma^T_C(z) \partial_{zz}^2 H_C(z) \sigma_C(z)\right ] - y^T u\,.
%\]

\begin{proposition}\label{PRO:ShapPCCasimirES}
Consider the closed--loop SPHS \eqref{EQN:SPHSInt1} and assume that it is possible to find $F$ and $S$ such that conditions \eqref{EQN:CondCasSC} hold. Assume further that $H_d(x) := H(x) + H_c(S^{-1}\left (c + F(x)\right ))$. 

If
\begin{description}[style=unboxed,leftmargin=0cm]
\item[(i) equilibrium assignment:] for a given $x_e \in \RR^n$, it holds
\begin{equation}\label{EQN:EqAssES2}
\partial_x H_c(S^{-1}(c + F(x_e))) = - \partial_x H(x_e)\,;
\end{equation}
\item[(ii) stability:] for $x_e \in \RR^n$, it holds
\begin{equation}\label{EQN:LStabES2}
\partial_x^2 H_c(S^{-1}(c + F(x_e))) > - \partial_x^2 H(x_e)\,;
\end{equation}
\item[(iii) ultimately stochastic passivity:] for all $x$ such that $\|x-x_e\|>C$,
\begin{equation}\label{EQN:LyapES}
\mathcal{L}H_d(x) \leq - \epsilon <0 \,;
\end{equation}
\item[(iv) uniqueness invariant measure:] either
\begin{description}[style=unboxed,leftmargin=0cm]
\item[(iv b) non-degeneracy:] the noise is non--degenerate, that is the matrix $\Sigma(x) := \sigma(x) \sigma^T(x) \succ 0$ is positive definite;
\item[(iv b) equivalence:] the transition kernel $p(t,x,A)$ is equivalent to the Lebesgue measure for any $t>0$ and $x \in \RR^n$.
\end{description}
\end{description}

Then there exists a unique invariant measure $\rho$ under which \eqref{EQN:SPHSInt1} is \textit{ultimately stochastic stable}. 
%and the invariant measure $\rho$ is peaked around $x^*$ and
%\[
%\rho_\sigma \to \delta_{x_e} \,,\quad \mbox{as} \quad \sigma \to 0\,.
%\]
\end{proposition}
\begin{proof}
The proof follows from above reasoning and proceeds similarly to the proof of Proposition \ref{PRO:ShapPC}.
\end{proof}

\section{Examples}\label{SEC:Example}

\subsection*{The stochastic inverted pendulum}

Consider an inverted pendulum with additive noise of the form
\begin{equation}\label{EQN:Pend}
\begin{cases}
d X^1(t) &= X^2(t) dt + dW^1(t) \,,\\
dX^2(t) &= \left (g \sin X^1(t) + u(t) \right )dt + dW^2(t)\,,
\end{cases}
\end{equation}
with $g$ denoting the gravitational acceleration and $u$ being the control. Defining
\[
J =
\begin{pmatrix}
0 & 1\\
-1 & 0
\end{pmatrix}\,,\quad R = 0 \,,\quad 
g = 
\begin{pmatrix}
0 & 1 
\end{pmatrix}^T\,,
\]
and using the Hamiltonian function
\begin{equation}\label{EQN:HamEx}
H(x_1,x_2)=\frac{1}{2}x_2^2 + g \cos x_1\,,
\end{equation}
system \eqref{EQN:Pend} can be seen to be a SPHS of the form \eqref{EQN:SPHSOr}.

We wish the system to converge towards a state $\bar{x}_e = (x_{1;e},0)$. Clearly, as the noise does not vanish at the state $\bar{x}_e$, there is no way convergence in probability can be achieved. We therefore consider a weak energy shaping setting, so that we will shape the limiting distribution around the desired state.

With the control law
\[
u = -x_2 - (x_1-x_{1;e}) - g \sin x_1\,,
\]
the system \eqref{EQN:Pend} is transformed into the system
\begin{equation}\label{EQN:ConvPend}
d X(t) =(J_s - R_s) \partial_x H_d(X(t)) dt +  dW(t)\,,\\
\end{equation}
with
\[
J_s =
\begin{pmatrix}
0 & 1\\
-1 & 0
\end{pmatrix}\,,\quad 
R_s = 
\begin{pmatrix}
0 & 0\\
0 & 1
\end{pmatrix}\,,\quad dW(t) = (dW^1(t),dW^2_(t))^T\,,
\]
and shaped Hamiltonian
\begin{equation}\label{EQN:QuadHamE}
H_d(x_1,x_2)=\frac{1}{2}x_2^2 + \frac{1}{2}(x_1 - x_{1;e})^2\,.
\end{equation}

A direct computation shows that
\[
\mathcal{L}H_d(x) = - x_2^2 + 1\,,
\]
so that, for $x_2^2 < 1 + \epsilon$, $\epsilon > 0$, the SPHS \eqref{EQN:ConvPend} is \textit{ultimately passive}. Further, since the Hamiltonian is quadratic, it can be proved that $X$ is a bivariate Gaussian random variable with invariant measure given by
\[
\rho(x) dx = N e^{- \frac{1}{2}\left [x_2^2 + (x_1-x_{1;e})^2\right ]}dx\,.
\]

An interesting consequence of the weak energy shaping theory is that, in order to obtain the desired convergence for this system, we are not forced to choose a control that compensates for the sinusoidal term in the Hamiltonian \eqref{EQN:HamEx}. In fact, choosing 
\[
u = -x_2 - (x_1-x_{1;e})\,,
\]
we obtain a system alike to \eqref{EQN:ConvPend} with Hamiltonian given by
\begin{equation}\label{EQN:HamShape}
H_d(x_1,x_2)=\frac{1}{2}x_2^2 + \frac{1}{2}(x_1 - x_{1;e})^2 + g \cos x_1\,.
\end{equation}

Hamiltonian \eqref{EQN:HamShape} is of the form of equation \eqref{EQN:QuadHamE} plus a potential given by $V(x^1)= g \cos x^1$. Explicit calculations implies that the ergodic invariant measure is
\[
\rho(x) dx = N e^{- \frac{1}{2}\left [x_2^2 + (x_1-x_{1;e})^2 - g (\cos x_1-1)\right ]}dx\,.
\]
Since we aims at shaping a limiting distribution peaked around $x_e$, we require
\[
\left .\partial_{x_1} \rho(x)\right |_{x_1 = x_{1;e}} = -N\left . \rho(x) \left (x_1-\tilde{x}_e + g (\cos x_1-1)\right )\right |_{x_1 = x_{1;e}} = 0\,,
\]
so that by setting
\[
\tilde{x}_e = x_{1;e} - g \sin x_{1;e}\,,
\]
we obtain the desired property.

\subsection*{The stochastic RLC circuit}

Consider the following SPHS
\begin{equation}\label{EQN:SRLC}
\begin{cases}
d X^1(t) &= \left [\alpha \partial_{x^2}H(X^1(t),X^2(t)) +  Eu(t)\right ] dt + \sqrt{2}\sigma^1 dW^1(t)\,,\\
d X^2(t) &= -\left [\alpha \partial_{x^1}H(X^1(t),X^2(t))+ \frac{1}{R_L}\partial_{x^2}H(X^1(t),X^2(t))\right ] dt + \sqrt{2} \sigma^2 dW^2(t)\,,\\
\end{cases}
\end{equation}
with Hamiltonian
\begin{equation}\label{EQN:HRLC}
H(x_1,x_2) = \frac{1}{2 L } x_1^2 + \frac{1}{2C}x_2^2\,,
\end{equation} 
where above, $X^1$ is the inductance flux, $X^2$ is the charge in the capacitor, $\alpha\in [0,1]$ represents the duty
ratio of the PWM, $R_L$ is the output load resistance and $E$ is the DC voltage source. Such system is the example considered in \cite{OrtS} with additive stochastic perturbation. The control objective is to drive the output capacitor voltage to some constant desired value $V_d > E$, maintaining internal stability. The equilibrium point is thus given by
\[
(x_{1;e},x_{2;e}) = \left (\frac{L V_d^2}{R_L E},CV_d\right )\,.
\]

Proceeding as in \cite[Section 7.A]{OrtS}, we can compute the feedback control as
\[
u = \phi(x) = - E\left (\frac{\frac{2}{R_L E} c_1 x_1 + c_2}{c_1 x_2 + c_3}\right )\,,
\]
where $c_1$, $c_2$ and $c_3$ are some constants that need to satisfy
\[
\begin{cases}
&\frac{R^2_L E^2}{4 L V^3_d} c_3 < c_1 < \frac{1}{C V_d}c_3\,,\quad c_3 < 0\,,\\
&-\frac{1}{C V_d}c_3 < c_1 <  \frac{R^2_L E^2}{4 L V^3_d} c_3 \,,\quad c_3 > 0\,,\\
&c_2 = - \left (\frac{2 L V_d^2 }{R^2_L V_d^2} + C\right )c_1 - \frac{1}{V_d}c_3\,.
\end{cases}
\]

The resulting shaped Hamiltonian is thus given by
\[
H_d(x_1,x_2) = \frac{1}{2 L }x_1^2 + \frac{1}{2C}x_2^2 + \frac{1}{2c_1}\frac{(c_1 x_2 + c_3)^2}{(\frac{2}{R_L E}c_1 x_1 + c_2)} - \frac{L V_d^4}{2 R^2_L E^2} + \frac{V_d c_3}{2 c_1}\,.
\]

To apply Proposition \ref{PRO:ShapPC} we must at last check that the shaped SPHS with Hamiltonian $H_d$ is ultimately stochastic passive. Computing the infinitesimal generator we thus have
\begin{equation}
\begin{split}
\mathcal{L}H_d(x_1,x_2) &= -\frac{1}{R_L}\left (\frac{1}{C}x_2 + \frac{c_1 x_2 + c_3}{\frac{2}{R_L E} c_1 x_1 + c_2}\right )^2 +\\
&+\left (\frac{1}{L} + \frac{4 c_1(c_1 x_2 +c_3)}{R^2_L E^2\left (\frac{2}{R_L E}c_1 x_1 + c_2\right )^3}\right )(\sigma^1)^2 +\\
&+ \left (\frac{1}{C} + \frac{c_1}{\frac{2}{R_L E}c_1 x_1 + c_2}\right )(\sigma^2)^2\,.
\end{split}
\end{equation}

Since we have that
\[
\mathcal{L}H_d (x_{1;e},x_{2;e}) =(\sigma^1)^2 \left ( \frac{1}{L} - \frac{4 c_1 V^3_d}{E^2 R_L^2 (c_1 C V_d + c_3)^2}\right ) + (\sigma^1)^2 \left (\frac{c_3}{c_3 C + c_1 C^2 V_d}\right )<\infty\,,
\]
using the fact
\[
\lim_{(x_1,x_2) \to \pm \infty}\mathcal{L}H_d(x_1,x_2) = - \infty\,,
\]
we can infer that there exist $C> 0$ and $\delta_C > 0$ such that for $\| x-(x_{1;e},x_{2;e}) \| \geq C$ it holds
\[
\mathcal{L}H_d(x_1,x_2) \leq - \epsilon < 0 \,.
\]

We can thus conclude that the shaped SPHS is ultimately stochastic passive and thus Proposition \ref{PRO:ShapPC} holds.

It is worth stressing that, given the non-linear potential appearing in the third term of the shaped Hamiltonian $H_d$, it is not possible to compute explicitly the invariant measure for the shaped SPHS. Nonetheless, Proposition \ref{PRO:ShapPC} still apply and we can conclude the existence and uniqueness of the invariant measure of the shaped SPHS. Similarly to what happens in the deterministic case, the strength of Proposition \ref{PRO:ShapPC} is that there is no need to explicitly compute the invariant measure to infer the convergence of the shaped SPHS. Also Remark \ref{REM:Rem} implies that the invariant measure of the shaped SPHS is peaked around the equilibrium we would obtain in the deterministic case with null noise.

\section{Conclusions}

The present paper continues the investigation of stochastic PHS started in \cite{CDPM_Tank,CDPM_IFAC,CDPMSPHS,CDPM_Dis,CDPM_Bil}, addressing the problem of energy shaping of SPHS. Such topic has been one of the main interest in the study of deterministic PHS and consequently it has been object of deep investigation also in the stochastic case. Nonetheless, existing results lack of a suitable generality to include relevant examples. In particular, stochastic systems with additive noise are often ruled out from the possible applications of energy shaping due to a non--vanishing noise. In the present paper we therefore generalizes the energy shaping techniques for SPHS. In particular, we introduce a weak notion of passivity, related to the ergodicity and the connected invariant measure for the SPHS. Such definition naturally leads to a weak notion of convergence in terms of transition semigroup. Compared to existing approaches to energy shaping, the one proposed in the current work has a purely probabilistic flavour, where the main objects are the invariant measure of the system and the related transition probabilities. At last, reformulating the problem of stochastic weak energy shaping in terms of the associated Fokker-Planck equation, energy shaping of deterministic infinite--dimensional port--controlled PHS is recovered, highlighting an insightful connection between the stochastic and deterministic energy shaping techniques.

\cleardoublepage

\cleardoublepage
%\nocite{*}
\bibliographystyle{apalike}
\bibliography{bib}
\end{document}